\documentclass[12pt,leqno]{article}
\usepackage{amsmath,amssymb,amsthm,amscd,dsfont}
\numberwithin{equation}{section} \allowdisplaybreaks
\begin{document}
\newtheorem{theorem}{Theorem}[section]
\newtheorem{defin}{Definition}[section]
\newtheorem{prop}{Proposition}[section]
\newtheorem{corol}{Corollary}[section]
\newtheorem{lemma}{Lemma}[section]
\newtheorem{rem}{Remark}[section]
\newtheorem{example}{Example}[section]
%\label{} %\ref{}
\title{On hypersurfaces of generalized K\"ahler manifolds}
\author{{\small by}\vspace{2mm}\\Izu Vaisman}
\date{}
\maketitle
{\def\thefootnote{*}\footnotetext[1]%
{{\it 2010 Mathematics Subject Classification: 53C15} .
\newline\indent{\it Key words and phrases}: generalized K\"ahler structure, generalized CRFK structure, generalized almost contact structure.}}
\begin{center} \begin{minipage}{12cm}
A{\footnotesize BSTRACT. We establish the conditions for the induced generalized metric F structure
of an oriented hypersurface of a generalized K\"ahler manifold to be a generalized CRFK structure. Then, we discuss a notion of generalized almost contact structure on a manifold $M$ that is suggested by the induced structure of a hypersurface. Such a structure has an associated generalized almost complex structure on $M\times\mathds{R}$. If the latter is integrable, the former is normal and we give the corresponding characterization. If the structure on $M\times\mathds{R}$ is generalized K\"ahler, the structure on $M$ is said to be binormal. We characterize binormality and give an example of binormal structure.}
\end{minipage}
\end{center} \vspace{5mm}
%\noindent
%begin{center} %\section %\end{center}
\section{Introduction}
The framework of this note is the $C^\infty$ category and the notation is classical \cite{KN}. We refer the reader to \cite{Bl} for the almost contact geometry and to \cite{Gualt} for generalized geometry. As a precaution, the differentiable manifolds that we consider are assumed to be connected and the submanifolds are connected embedded submanifolds.

It is well known \cite{{Bl},{T}} that, if $M^{2n}$ is an almost Hermitian manifold and $\iota:N\subseteq M$ is an oriented hypersurface in $M$, one can define a corresponding induced, metric, almost contact structure of $N$. This structure was studied by many authors. In particular, conditions for the induced structure to be normal were given in \cite{{Ok},{T},{YI}}.

In the present paper we study similar problems in generalized geometry. We show that a generalized almost Hermitian structure on $M$ induces a generalized metric F structure on the hypersurface $N$ and we discuss questions related to the integrability of the latter. In the classical situation, we find conditions for the induced structure to be classical CRF in the sense of \cite{V1}, which is a weaker property than normality. Then, we consider the case where $M$ is a generalized K\"ahler manifold and we establish the conditions for the induced structure to be a generalized CRFK structure in the sense of \cite{V1}. In the previous terminology, CR stands for Cauchy-Riemann (e.g., \cite{Bl}, F stands for the notion of an F structure \cite{Y0} and K indicates a connection with K\"ahler geometry.

The induced structure of a hypersurface of a generalized almost Hermitian manifold belongs to a particular class of generalized F structures of a manifold $M$ that naturally extend the classical almost contact structures to the generalized framework. This extension is different from the generalized almost contact structures of \cite{{IW},{PW},{V3}} but, it is equivalent to those of Sekiya \cite{Sek}. Using an associated generalized almost complex structure on $M\times\mathds{R}$, we define corresponding notions of normality and binormality and establish the conditions that characterize normal and binormal structures. The notion of a generalized binormal structure is equivalent to that of generalized coK\"ahler structure studied in \cite{GT}.

In the whole paper we use the following important notational convention: formulas that include the double signs $\pm,\,\mp$ contain two cases, the case of all the upper signs and the case of all the lower signs; mixing between upper and lower signs is not included.
\section{Preliminaries}
In this section we recall the definition of some of the structures that we shall encounter later on.
{\it Classical structures} are defined on $TM$ and {\it generalized structures} are defined on the {\it big tangent bundle} $\mathbf{T}M=TM\oplus T^*M$. $TM$ is a {\it Lie algebroid} with the anchor $Id:TM\rightarrow TM$ and the Lie bracket $[\,,\,]$, while $\mathbf{T}M$ is a {\it Courant algebroid} (e.g., \cite{Gualt}) with the anchor $pr_{TM}$, the neutral pairing metric
$$g((X,\alpha),(Y,\beta))=\frac{1}{2}(\alpha(Y)+\beta(X))\;
(X,Y\in TM,\alpha,\beta\in T^*M)$$
and the Courant bracket
$$
[\mathcal{X},\mathcal{Y}]=([X,Y],(L_X\beta-L_Y\alpha+\frac{1}{2}d(\alpha(Y)-\beta(X))).
$$
Calligraphic characters denote pairs $\mathcal{X}=(X,\alpha),\mathcal{Y=}(Y,\beta)$, etc. Among the properties of the Courant bracket we notice
\begin{equation}\label{propofC}
[\mathcal{X},f\mathcal{Y}]=f[\mathcal{X},\mathcal{Y}]+pr_{TM}\mathcal{X}(f)\mathcal{Y} -g(\mathcal{X},\mathcal{Y})\partial f,\end{equation}
where $f\in C^\infty(M)$ and $\partial f$ is defined by $g(\partial f, \mathcal{U})=(1/2)pr_{TM}\mathcal{U}(f)$, $\forall\mathcal{U}\in\mathbf{T}M$.

A classical almost contact structure is a triple $(F\in End(TM),Z\in\chi(M),\xi\in\Omega^1(M))$ such that
\begin{equation}\label{almcont} F^2=-Id+\xi\otimes Z,\;FZ=0,\; \xi\circ F=0,\;\xi(Z)=1. \end{equation}
The almost contact structure $(F,Z,\xi)$ on $M$ is equivalent with the almost complex structure defined on $M\times\mathds{R}$ by
\begin{equation}\label{JF}
J=F-dt\otimes Z+\xi\otimes\frac{\partial}{\partial t},\;\; t\in\mathds{R}.\end{equation}

A tensor field $F\in End(TM)$ is an F structure if $F^3+F=0$. In particular,	(\ref{almcont}) implies that, if $(F,Z,\xi)$ is an almost contact structure, then, $F$ is an F structure.

Furthermore, $(F,Z,\xi,s)$, where $s$ is a Riemannian metric, is a metric almost contact structure if
\begin{equation}\label{clasmetric}s(FX,FY)=s(X,Y)-\xi(X)\xi(Y).\end{equation} Condition (\ref{clasmetric}) implies $Z\perp_s\,im\,F$, $\xi=\flat_sZ$, $||Z||_s=1$ and
\begin{equation}\label{Fmetric} s(FX,Y)+s(X,FY)=0,\end{equation}
where the musical isomorphism has the usual definition $(\flat_sZ)(X)=s(Z,X)$.
A pair $(F,s)$, where $F$ is an F structure and $s$ a Riemannian metric is a metric F structure if (\ref{Fmetric}) holds.
Then, $\Xi(X,Y)=s(FX,Y)$ is a $2$-form called the fundamental form.

An almost contact structure of $M$ is normal iff the corresponding almost complex structure of $M\times\mathds{R}$ is integrable and the normality condition is\footnote{For exterior products and differentials we use Cartan's convention, e.g.,
$$\alpha\wedge\beta(X,Y)=\alpha(X)\beta(Y)-\alpha(Y)\beta(X),
\;d\alpha(X,Y)=X\alpha(Y)-Y\alpha(X)-\alpha([X,Y]).$$}
\begin{equation}\label{normal} \mathcal{N}_F+d\xi\otimes Z=0,\end{equation}
where $\mathcal{N}_F$ is the Nijenhuis tensor
\begin{equation}\label{Nijtens}
\mathcal{N}_F(X,Y)=[FX,FY]-F[FX,Y]-F[X,FY]+F^2[X,Y].\end{equation}

An F structure $F$ has the eigenvalues $\pm i,0$ with corresponding eigenbundles $H,\bar{H},Q=ker\,F$ such that $T_cM=H\oplus\bar{H}\oplus Q_c$, $im\,F=P$ where $P_c=H\oplus\bar{H}$ (the index $c$ denotes complexification). The projections of $T_cM$ on the eigenbundles are given by
\begin{equation}\label{proneigen} pr_H=-\frac{1}{2}(F^2+iF),\,pr_{\bar{H}}=-\frac{1}{2}(F^2-iF),\, pr_Q=Id+F^2,\,pr_P=-F^2.\end{equation}

If the eigenbundle $H$ is involutive, it defines a CR structure and the F structure (or the almost contact structure that includes $F$, if given) are said to be of the CR type. Using eigenvectors as arguments, we see that the CR type condition is
\begin{equation}\label{CRcond} N_F(X,Y)=pr_Q[X,Y],\;\forall X,Y\in P.\end{equation}
In the case of an almost contact structure, normality implies CR type.

A {\it generalized almost complex structure} is $\mathcal{J}\in End(\mathbf{T}M)$ such that
$$
g(\mathcal{J}\mathcal{X},\mathcal{Y}) +
g(\mathcal{X},\mathcal{J}\mathcal{Y})=0,\;
\mathcal{J}^2=-Id,\;
\forall\mathcal{X},\mathcal{Y}\in\Gamma \mathbf{T}M
$$ ($\Gamma$ denotes the space of sections of a vector bundle).
The structure is {\it integrable} or {\it generalized complex} if $\mathcal{N}_{\mathcal{J}}(\mathcal{X},\mathcal{Y})=0$, where the Nijenhuis tensor is also ''generalized", i.e., defined by (\ref{Nijtens}) with Courant brackets instead of Lie brackets. The integrability condition is equivalent to the closure of $\Gamma(\mathcal{L})$ under Courant brackets, where $\mathcal{L}$ is the $i$-eigenbundle of $\mathcal{J}$.

A {\it generalized F structure} is $\mathcal{F}\in End(\mathbf{T}M)$ such that
$$
g(\mathcal{F}\mathcal{X},\mathcal{Y}) +
g(\mathcal{X},\mathcal{F}\mathcal{Y})=0,\;
\mathcal{F}^3+\mathcal{F}=0.$$
Then, $\mathcal{F}$ has the eigenvalues $\pm i,0$ with corresponding eigenbundles $E,\bar{E},S=ker\,\mathcal{F}$. We also consider the subbundle $L=im\,\mathcal{F}\subseteq \mathbf{T}M$. Obviously, $L_c=E\oplus\bar{E}$, $L\cap S=0$, $L\perp_g S$, $\mathbf{T}_cM=E\oplus\bar{E}\oplus S_c=(L\oplus S)_c$ and the restrictions of $g$ to $L$ and $S$ are non degenerate. The projections are now given by (\ref{proneigen}) with $F\mapsto\mathcal{F},H\mapsto E,Q\mapsto S,P\mapsto L$. The rank of the vector bundle $S$ is called the {\it corank} of $\mathcal{F}$ and the negative inertia index of $g|_S$ is called the {\it negative index} $neg\,\mathcal{F}$ of $\mathcal{F}$.

Integrability of the generalized F structure is defined by the requirement that $\Gamma(E)$ is closed under Courant brackets and, if this happens, the structure is called a {\it CRF structure}. The integrability condition is equivalent to \cite{V1}
$$ \mathcal{N}_{\mathcal{F}}(\mathcal{X},\mathcal{Y})=
pr_S[\mathcal{X},\mathcal{Y}],\;\forall\mathcal{X},\mathcal{Y}\in\Gamma L. $$

If $\mathcal{F}(X,\alpha)=(FX,-\alpha\circ F)$, where $F$ is an F structure on $M$ and if integrability holds, $F$ is a {\it classical CRF structure}. The classical CRF condition is equivalent to \cite{V1}
\begin{equation}\label{integrcuH} [H,H]\subseteq H,\,[H,Q_c]\subseteq H\oplus Q_c,\end{equation}
where the brackets are Lie brackets.
The first condition (\ref{integrcuH}) shows that $F$ is of the CR type, hence, condition (\ref{CRcond}) holds. Furthermore \cite{V1}, $F$ is classical CRF iff it is of the CR type and
\begin{equation}\label{CRF0} \mathcal{N}_F(X,Y)=0,\;\forall X\in P, Y\in Q.\end{equation}
(All these conditions are easily checked using eigenvectors as arguments.)

The previous definitions may be applied to an almost contact structure $(F,Z,\xi)$ asking the CR, respectively CRF, condition to hold for the component $F$ of the structure. Then, $Q=span\{Z\}$ and condition (\ref{CRF0}) becomes
\begin{equation}\label{CRFcuLie} F[Z,FX]-F^2[Z,X]=0\,\Leftrightarrow\,F\circ(L_ZF)=0\;(X\in TM).
\end{equation}

Condition (\ref{CRFcuLie}) has the following geometric interpretation. $F|_P$ defines complex structures $\tilde{F}_x$ in $T_xM/span\{Z_x\}$, $\forall x\in N$. These structures yield a transversal almost complex structure $\tilde{F}$ of the foliation defined by the trajectories of $Z$ iff the tensor field $F$ is projectable, meaning that for every projectable vector field $X\in\chi(M)$ (one such that $[X,Z]\in span\{Z\})$ the vector field $FX$ is again projectable ($[FX,Z]\in span\{Z\}$).
It	easily follows that projectability of $F$ is equivalent to (\ref{CRFcuLie}). Furthermore, by using projectable vector fields as arguments, we see that the CR condition is equivalent to the integrability of $\tilde{F}$. We conclude that an almost contact structure $(F,Z,\xi)$ is classical CRF iff $F$ defines an induced transversal, complex (holomorphic) structure of the foliation $span\{Z\}$.

Alternatively, the almost contact structure $(F,Z,\xi)$ of $M$ is equivalent to the almost complex structure $J$ given by (\ref{JF}) and $F$ is classical CRF iff $J$ induces a transversal holomorphic structure of the foliation $span\{Z,\partial/\partial t\}$.
\begin{rem}\label{obsctnormal} {\rm Taking separately in the normality condition (\ref{normal}) the cases $X,Y\in P$ and $X\in P,Y=Z$, it follows that an almost contact structure is normal iff it is of the CR type and $L_ZF=0$. Therefore, a normal structure is classical CRF.}\end{rem}

We shall also need the notion of a {\it generalized Riemannian metric} on a manifold $M$ \cite{Gualt}. This is a Euclidean
(positive definite) metric $G$ on $\mathbf{T}M$ together with a $G$-orthogonal
decomposition
$$\mathbf{T}M=V_+\oplus V_-,$$
where $V_\pm$ are maximal $g$-positive, respectively $g$-negative, $g$-orthogonal subbundles
and $G$ is equal to $\pm g$ on $V_\pm$. (In fact, the decomposition defines the metric by putting $G=g|_{V_+}-g|_{V_-}$.)

In \cite{Gualt}, it was shown that $G$ is equivalent with a pair
$(\gamma,\psi)$, where $\gamma$ is a usual Riemannian metric on $M$
and $\psi\in\Omega^2(M)$. This equivalence is realized by putting
\begin{equation}\label{exprEpm}
V_{\pm}=\{(X,\flat_{\psi\pm\gamma}X)\,/\,X\in TM\}.\end{equation}
Formula (\ref{exprEpm}) shows the existence of the isomorphisms
$\tau_\pm:V_\pm\rightarrow TM$ given by
$\tau_\pm(X,\flat_{\psi\pm\gamma}X)=X$, and $\tau_\pm$ may be used to
transfer structures between $V_\pm$ and $TM$. In particular, the two
metrics $G|_{V_\pm}$ transfer to $\gamma$. The bundles $V_\pm$ may be seen as the $\pm1$-eigenbundles of an isomorphism $\mathcal{G}:\mathbf{T}M\rightarrow\mathbf{T}M$, which satisfies the conditions
\begin{equation}\label{condptGrond}\begin{array}{c}\mathcal{G}^2=Id, \; g(\mathcal{G}\mathcal{X},\mathcal{G}\mathcal{Y})=g(\mathcal{X},\mathcal{Y}) \vspace*{2mm}\\ G(\mathcal{G}\mathcal{X},\mathcal{Y})=g(\mathcal{X},\mathcal{Y}). \end{array}\end{equation}
Conversely, if $\mathcal{G}$ satisfies the first two conditions (\ref{condptGrond}), $\mathcal{G}$ has the eigenvalues $\pm1$ and two $g$-orthogonal eigenbundles $V_\pm$, where the restriction of $g$ is non degenerate. Furthermore, if $g|_{V_+}$ is positive definite, the second formula (\ref{condptGrond}) yields a generalized Riemannian metric $G$.

A {\it generalized almost Hermitian structure} is a pair $(G,\mathcal{J})$, where $G$ is a generalized Riemannian metric, $\mathcal{J}$ is a generalized almost complex structure and
$$ G(\mathcal{J}\mathcal{X},\mathcal{J}\mathcal{Y})= G(\mathcal{X},\mathcal{Y})\;\;(\mathcal{X},\mathcal{Y}\in\mathbf{T}M).
$$
Then, $\mathcal{J}$ defines the following two classical almost complex structures on $M$
\begin{equation}\label{eqtransfer}J_\pm=\tau_\pm\circ\mathcal{J}|_{V_\pm}\circ\tau_\pm^{-1}
\end{equation}
that are compatible with the metric $\gamma$ given by $G\Leftrightarrow(\gamma,\psi)$ and one has an equivalence between $(G,\mathcal{J})$ and the quadruple $(\gamma,\psi,J_\pm)$ \cite{Gualt}. In particular, formula (\ref{eqtransfer}) yields the expression of $\mathcal{J}$
\begin{equation}\label{eqJrond} \mathcal{J}(X,\flat_{\psi\pm\gamma}X)=(J_\pm X,\flat_{\psi\pm\gamma}J_\pm X).\end{equation}

If the structure $\mathcal{J}$ is integrable, $(G,\mathcal{J})$ is a {\it generalized Hermitian structure}. A more interesting structure \cite{{Gualt},{V3}} is obtained by noticing that
the pair $(G,\mathcal{J})$ produces a {\it complementary} generalized almost Hermitian structure $(G,\mathcal{J}')$ associated to the quadruple $(\gamma,\psi,J_+,-J_-)$. In view of (\ref{eqtransfer}) $\mathcal{J}\circ\mathcal{J}'=\mathcal{J}'\circ\mathcal{J}$, which implies that the $i$-eigenbundles $\mathcal{L},\mathcal{L}'$ of $\mathcal{J},\mathcal{J}'$ are
$$\mathcal{L}=(\mathcal{L}\cap V_+)\oplus(\mathcal{L}\cap V_-),\,\mathcal{L}'=(\mathcal{L}\cap V_+)\oplus(\bar{\mathcal{L}}\cap V_+).$$
The structure $(G,\mathcal{J},\mathcal{J}')$ is a {\it generalized K\"ahler structure} if the two structures $\mathcal{J},\mathcal{J}'$ are integrable and the generalized K\"ahler condition is equivalent to the following two properties: (i) the pairs $(\gamma,J_\pm)$ are Hermitian structures, (ii) for the Hermitian structures $(\gamma,J_\pm)$ one has
\begin{equation}\label{relpsiJ} \gamma(\nabla^\gamma_XJ_\pm(Y),U)=\mp\frac{1}{2}[d\psi(X,J_\pm Y,U)+d\psi(X,Y,J_\pm U)],
\end{equation}
where $\nabla^\gamma$ is the Levi-Civita connection of $\gamma$.

For generalized F structures the corresponding notions were discussed in \cite{V1}. A pair $(G,\mathcal{F})$, where $\mathcal{F}$ is a generalized F structure and $G$ is  a generalized Riemannian metric, is a {\it generalized metric F structure} if
\begin{equation}\label{G-F}
G(\mathcal{F}\mathcal{X},\mathcal{Y})+G(\mathcal{X},\mathcal{F}\mathcal{Y})=0.\end{equation}
As in the case of generalized almost Hermitian structures, it follows that the pair $(G,\mathcal{F})$ is equivalent to a quadruple $(\gamma,\psi,F_\pm)$, where $(\gamma,F_\pm)$ are classical metric F structures (i.e., $\gamma(FX,Y)+\gamma(X,FY)=0$). The endomorphism $\mathcal{F}$ is expressed by (\ref{eqJrond}) with $F_\pm$ instead of $J_\pm$.

Furthermore, the quadruple $(\gamma,\psi,F_+,-F_-)$ defines a second generalized metric F structure $(G,\mathcal{F}')$ and $(G,\mathcal{F})$ is a {\it generalized CRFK structure}	 if $\mathcal{F},\mathcal{F}'$ are generalized CRF structures (both are integrable) and
$$[V_+\cap S,V_-\cap S]\subseteq S,$$ where the brackets are Courant brackets and $S=ker\,\mathcal{F}$ \cite{V1}.

In \cite{V1} it was proven that $(G,\mathcal{F})$ is a generalized CRFK structure iff: (1) the corresponding structures $F_\pm$ are classical CRF structures, (2) one has the equalities
\begin{equation}\label{CRFK6} \gamma(F_\pm\nabla^\gamma_XF_\pm(Y),U)= \pm\frac{1}{2}[d\psi(X,Y,F^2_\pm U)+d\psi(X,F_\pm Y,F_\pm U)],\end{equation}
$\forall X,Y,U\in TM$.
\begin{example}\label{CRFKclasic} {\rm A classical, metric F structure $(F,s)$ may be identified with the generalized structure defined by the quadruple $(F_+=F_-=F,\gamma=s,\psi=0)$. This structure is CRFK iff $F$ is classical CRF and, $\forall X,Y\in TM$, $\nabla^s_XF(Y)\in ker\,F$. Then, we will say that $(M,F,s)$ is a {\it classical CRFK manifold}. For instance, in the case of a cosymplectic manifold in the sense of Blair, i.e., a metric almost contact manifold $(F,Z,\xi,s)$ such that $\nabla^s F=0$, $(M,F,s)$ is classical CRFK. Indeed, the above kernel condition holds trivially. On the other hand, we also have $\nabla^s Z=0$ \cite{Bl}, which implies $L_ZF=0$, therefore, the structure $F$ is classical CRF.}\end{example}
\begin{rem}\label{partialK} {\rm Condition (\ref{relpsiJ}) shows that a generalized K\"ahler manifold with a closed form $\psi$ is a {\it bi-K\"ahlerian manifold}, i.e., the two Hermitian structures $(\gamma,J_\pm)$ are K\"ahler structures. Similarly, Proposition 4.7 of \cite{V1} shows that a generalized CRFK manifold with a closed form $\psi$ is a {\it partially bi-K\"ahlerian manifold}, meaning that $\gamma$ has two de Rham decompositions into a sum of metrics where one term is K\"ahlerian.}\end{rem}
\section{Induced structures of hypersurfaces}
Let $M^{2n}$ be an almost Hermitian manifold with the almost complex structure $J$ and the metric $\gamma$. Let $\iota:N\subseteq M$ be an oriented hypersurface with the unit normal vector field $\nu$. Then, the decomposition
\begin{equation}\label{strind1}JX=FX+\xi(X)\nu,\; X\in TN,\, F\in End(TN),\xi\in\Omega^1(N)\end{equation}
defines an almost contact structure $(F,\xi,Z=-J\nu)$ on $N$.
Moreover, if $s=\iota^*\gamma$, $(F,Z,\xi,s)$ is a metric almost contact structure on $N$
and its fundamental form $\Xi(X,Y)=s(FX,Y)$ is given by
$\Xi=\iota^*\Omega$, where $\Omega(X,Y)=\gamma(JX,Y)$ is the K\"ahler form of $(\gamma,J)$ \cite{{Bl},{T}}.

The definition of the induced structure yields the following simple result.
\begin{prop}\label{prop1Killing} If $(M,J,\gamma)$ is a Hermitian manifold and if the hypersurface $N$ is orthogonal to a holomorphic Killing unit vector field $U$ of $M$, then, the induced metric almost contact structure of $N$ is normal.\end{prop}
\begin{proof}
The normality condition (\ref{normal}) is a local type condition on $N$. Hence, we may prove the proposition by a local type argument. Since $U$ is nowhere zero and $N$ is transversal to $U$, $\forall x\in N$, there exists an open neighborhood $\mathcal{V}_x\subseteq M$ endowed with local coordinates $(x^u,t)$, $u=1,...,2n-1$, such that $U|_{\mathcal{V}_x}=\partial/\partial t$ and $\mathcal{U}_x=N\cap\mathcal{V}_x$ has the equation $t=0$. Thus, we may assume that $\mathcal{V}_x\approx\mathcal{U}_x\times(-\epsilon_x,\epsilon_x)$ ($0<\epsilon_x\in\mathds{R}$).
The condition $||U||_\gamma=1$ means that $t$ is the arc length on the trajectories of $U$ and it implies $\mu=\flat_\gamma U=dt$.

The definition (\ref{strind1}) of the induced structures implies that the expression of $J$ on $\mathcal{U}_x$ is exactly (\ref{JF}) with $\mathds{R}$ replaced by the interval $(-\epsilon_x,\epsilon_x)$ (check for $JX$ with $X\in TN$ and for $J(\partial/\partial t)$).
Then, since the flow of $U$ preserves $J$ and $\gamma$, if we apply $exp(tU)$ (using its expression in the local coordinates $(x^u,t)$) to (\ref{JF}), we see that (\ref{JF}) holds on $\mathcal{U}_x\times(-\epsilon_x,\epsilon_x)$. Obviously, the integrability of $J$ implies (\ref{normal}) on $\mathcal{U}_x$ just as in the case of $N\times\mathds{R}$.

Therefore, each point $x\in N$ has a neighborhood where the structure of $N$ is normal and we are done.
\end{proof}

Killing vector fields were studied intensively in Riemannian, complex and K\"ahler geometry. For instance, it is known that the trajectories of a unit Killing vector field are geodesic lines (e.g., see \cite{BN} for an extensive study of constant length Killing vector fields) and that on	a compact K\"ahler manifold a Killing vector field must be holomorphic (e.g., \cite{Mor}).

We shall deduce the conditions for the induced structure of a hypersurface to be a classical CRF structure.

The $i$-eigenbundle of $F$ defined by (\ref{strind1}) is $H=W\cap(T^cN)$,
where $W$ is the $i$-eigenbundle of $J$. Therefore, if $J$ is integrable, $H$ is closed under the Lie bracket and the structure $F$ is of the CR type. Accordingly, since CRF requires CR type, we shall restrict our discussion to hypersurfaces of Hermitian manifolds.

We recall the Gauss-Weingarten equations of the hypersurface $N$ with unit normal field $\nu$ \cite{KN}:
\begin{equation}\label{G-W} \nabla^\gamma_XY=\nabla^s_XY+b(X,Y)\nu,\;\nabla^\gamma_X\nu=-W_\nu X+\nabla^\nu_X\nu,\end{equation} where $X,Y\in TN$, $\nabla^\gamma,\nabla^s$ are the Levi-Civita connections of the metrics $\gamma,s$,
$\nabla^\nu$ is the induced connection of the normal bundle of $N$ and $b,W_\nu$ are the second fundamental form and Weingarten operator, respectively. The latter are related by the formula $s(W_\nu X,Y)=b(X,Y)$. In fact, since $\nu$ is of length $1$, $\nabla^\nu_X\nu=0$, but, we wrote it anyway since we look at equations (\ref{G-W}) as defining a connection on the vector bundle $T_NM$ over $N$, which we call the {\it Gauss-Weingarten connection}.

The conditions for the induced structure to be classical CRF and to be normal are
given by
\begin{prop}\label{answer2} 1. The induced structure $(F,Z,\xi,s)$ of the oriented hypersurface $N$ of the Hermitian manifold $(M,\gamma,J)$ is a classical CRF structure iff, $\forall X,Y\in TN$ such that $\Omega(X,\nu)=0,\Omega(Y,\nu)=0$
(equivalently, $X,Y\in P=im\,F$), one has
\begin{equation}\label{eqCRF2} d\Omega(JX,JY,J\nu)=d\Omega(X,Y,J\nu),\;
b(FX,FY)=b(X,Y).\end{equation}
2. The same induced structure is normal iff (\ref{eqCRF2}) holds and \begin{equation}\label{eqnormal2}b(Z,X)=b(X,Z)=-\frac{1}{2}d\Omega(\nu,Z,JX),
\;\;\forall X\in P.\end{equation}
\end{prop}
\begin{proof}
The CRF condition (\ref{CRFcuLie}) identically holds for $X=Z$, hence, we have to ask
$F(L_ZF)(X)=0$, $\forall X\in P$, equivalently,
\begin{equation}\label{auxanswer2} s(FL_ZF(X),Y)=-s(L_ZF(X),FY)=0,\;\forall X,Y\in P
\end{equation} (the same equality obviously holds for $Y=Z$).

Since $\nabla^s$ has no torsion, we have
\begin{equation}\label{LFcunabla0}
L_ZF(X)=[Z,FX]-F[Z,X]=\nabla^s_ZF(X)-\nabla^s_{FX}Z+
F\nabla^s_XZ.\end{equation}

Using the Gauss-Weingarten connection and the definition (\ref{strind1}) of the induced structure, we get, $\forall X,U\in TN$,
\begin{equation}\label{legLieFJ} \nabla^s_XF(U)=\nabla^\gamma_XJ(U)+\xi(U)W_\nu X-b(X,U)Z-(\nabla^s_X\xi)(U)\nu.\end{equation}

Using (\ref{legLieFJ}) for $X\mapsto Z,U\mapsto X\in P$ and$\xi(X)=0,\xi(Z)=1$, (\ref{LFcunabla0}) yields
$$
 L_ZF(X)=\nabla^\gamma_ZJ(X)-\nabla^\gamma_{JX}Z+ J\nabla^\gamma_XZ+(b(FX,Z)+d\xi(X,Z))\nu.
$$

Accordingly and since $Y\in P$ implies $JY=FY$, condition (\ref{auxanswer2}) becomes
\begin{equation}\label{auxanswer20} \gamma(\nabla^\gamma_ZJ(X),JY)-\gamma(\nabla^\gamma_{JX}Z,JY)
+\gamma(\nabla^\gamma_XZ,Y)=0,\,\forall X,Y\in P.\end{equation}
Then, using $Z=-J\nu$, the Weingarten equation and the relation between $W_\nu$ and $b$, (\ref{auxanswer20}) changes to
\begin{equation}\label{auxanswer21}\begin{array}{l} \gamma(\nabla^\gamma_ZJ(X),JY)+\gamma(\nabla^\gamma_{JX}J(\nu),JY)- \gamma(\nabla^\gamma_XJ(\nu),Y)
\vspace*{2mm}\\ \hspace*{1cm}=b(FX,Y)+b(X,FY).\end{array}\end{equation}

Furthermore, it is known that the integrability of $J$ is equivalent to the following equality (Proposition IX.4.2 of \cite{KN} with our sign convention for $\Omega$ and our convention on the evaluation of the exterior differential)
\begin{equation}\label{eqdinKN}
2\gamma(\nabla^\gamma_XJ(Y),U)=d\Omega(X,Y,U)-d\Omega(X,JY,JU),\,\forall X,Y,U\in TM.\end{equation}
Using this equality, (\ref{auxanswer21}) becomes
\begin{equation}\label{auxanswer22}\begin{array}{l} d\Omega(Z,X,JY)+d\Omega(JX,Y,Z)
+\frac{1}{2}(d\Omega(\nu,X,Y)\vspace*{2mm}\\ -d\Omega(\nu,JX,JY))
=b(FX,Y)+b(X,FY),\;\forall X,Y\in P.\end{array}\end{equation}

In (\ref{auxanswer22}) the left hand side is skew symmetric in $X,Y$ and the right hand side is symmetric, therefore, the condition holds iff the two sides vanish. The vanishing of the right hand side becomes the second condition (\ref{eqCRF2}) by putting $FY$ instead of $Y$. The vanishing of the left hand side becomes the first condition (\ref{eqCRF2}) if we replace $\nu=JZ$, use the identity
\begin{equation}\label{identHerm}
d\Omega(JZ,JX,JY)=d\Omega(JZ,X,Y)+d\Omega(Z,JX,Y)+d\Omega(Z,X,JY),\end{equation}
which holds for any vector fields on a Hermitian manifold, and, in the end, replace $Y$ by $JY$. (Identity (\ref{identHerm}) may be checked on any combination of arguments of the complex type $(1,0),(0,1)$ while recalling that $d\Omega$ has only terms of type $(2,1),(1,2)$.)

For the normality of the induced structure we have to ask $L_ZF=0$, which means that we have to add to (\ref{auxanswer2}) the condition $s(L_ZF(X),Z)=0$ for $X\in P$. In view of (\ref{LFcunabla0}) the result is
$$s(\nabla^s_Z(FX),Z)-s(\nabla^s_{FX}Z,Z)=0,\;\;\forall X\in P.$$
Here, the second term of the left hand side vanishes because $s(Z,Z)=1$. Furthermore, since $X\in P$, $FX=JX$, and using the Gauss equation, the previous condition takes the form $\gamma(\nabla^\gamma_Z(JX),Z)=0$, equivalently,
$$\gamma(\nabla^\gamma_ZJ(X),Z)- \gamma(\nabla^\gamma_ZX,\nu)=\gamma(\nabla^\gamma_ZJ(X),Z) -b(Z,X)=0.$$
Finally, the equality (\ref{eqdinKN}) shows that the required condition is (\ref{eqnormal2}).
\end{proof}
\begin{rem}\label{obsprimacond} {\rm The first condition (\ref{eqCRF2}) refers to the position of the hypersurface in the ambient manifold. But, it is equivalent to the following property of the fundamental form $\Xi$ of the induced structure:
\begin{equation}\label{LXi} L_Z\Xi(FX,FY)=L_Z\Xi(X,Y),\;\forall X,Y\in P. \end{equation}
Indeed, one has
$$\begin{array}{l}d\Omega(X,Y,J\nu)-d\Omega(JX,JY,J\nu)
=i(Z)d\Xi(FX,FY)-i(Z)d\Xi(X,Y)\vspace*{2mm}\\ \hspace*{25mm}=L_Z\Xi(FX,FY)-L_Z\Xi(X,Y),\end{array}$$
because $i(Z)\Xi=0$. Using $i(Z)\Xi=0$ we also see that, in fact, (\ref{LXi}) holds $\forall X,Y\in TN$.

We notice that property (\ref{LXi}) holds on any metric almost contact manifold that is classical CRF. Indeed, the explicit expression of $L_Z\Xi(FX,FY)$ yields
$$L_Z\Xi(FX,FY)=L_Z\Xi(X,Y) +\Xi(F(L_ZF(X),Y)+\Xi(X,FL_ZF(Y))$$ and $F\circ (L_ZF)=0$ in the CRF case.}\end{rem}
\begin{corol}\label{corolK} 1. The induced structure $(F,Z,\xi)$ of the oriented hypersurface $N$ of the K\"ahler manifold $(M,\gamma,J)$ is a classical CRF structure iff\footnote{The necessity of this condition shows that Proposition 3.6 of \cite{VAnn} is not correct. In fact, the error goes back to Proposition 2.5 of \cite{V1}.}
\begin{equation}\label{eqCRF3}	b(FX,FY)=b(X,Y),\;\;\forall X,Y\in P.\end{equation}
2. The same induced structure is normal iff (\ref{eqCRF3}) holds and, $\forall X\in P,$ $b(Z,X)=0$.
In particular, the normality conditions hold for totally geodesic and umbilical hypersurfaces.
\end{corol}
\begin{proof} In the K\"ahler case one has $d\Omega=0$, hence, the first condition (\ref{eqCRF2}) holds and (\ref{eqnormal2}) reduces to $b(Z,X)=0$. Totally geodesic means $b=0$. Umbilical means $b=\lambda s$ and the conditions hold because $s$ and $F$ are compatible. The second part of the corollary is well known \cite{{T},{Ok}}.\end{proof}
\begin{corol}\label{corolXi} If the induced structure of a hypersurface $N$ of a Hermitian manifold $M$ has a closed fundamental form $\Xi$, the structure is CRF iff (\ref{eqCRF3}) holds. In particular if $N$ is either totally geodesic or totally umbilical and $d\Xi=0$, the induced structure of $N$ is classical CRF.\end{corol}
\begin{proof} Use condition (\ref{LXi}).\end{proof}

Now, we shall consider the generalized framework.
If $\iota:N\subseteq M$ is an oriented hypersurface of a generalized almost Hermitian manifold$(M,G,\mathcal{J})\Leftrightarrow(M,\gamma,\psi,J_\pm)$ (see Section 2), N has an induced generalized Riemannian metric $G'$ defined by the pair $(s=\iota^*\gamma,\kappa=\iota^*\psi)$ and induced, metric, almost contact structures $(F_\pm,Z_\pm,\xi_\pm,s)$. In particular, the quadruple $(s,\kappa,F_\pm)$ defines an induced generalized $G'$-metric F structure $\mathcal{F}$ of $N$. We will discuss the conditions for this induced structure to be CRFK in the case where $M$ is a generalized K\"ahler manifold, i.e., $J_\pm$ are integrable and (\ref{relpsiJ}) holds. We shall also use the fact that, in view of
(\ref{eqdinKN}), condition (\ref{relpsiJ}) is equivalent to \cite{Gualt}
\begin{equation}\label{relpsiOmega} d\Omega_\pm(J_\pm X,J_\pm Y,J_\pm U)=\mp d\psi(X,Y,U).
\end{equation}

Firstly, the CRFK condition (1), Section 2, shows that the two induced structures $F_\pm$ have to be classical CRF structures, i.e., (see (\ref{eqCRF2})), if $X,Y\in P_\pm$, then,
\begin{equation}\label{FpmclasCRF} d\Omega_\pm(J_\pm X,J\pm Y,J_\pm\nu)=d\Omega_\pm(X,Y,J_\pm\nu),\;
b(F_\pm X,F_\pm Y)=b(X,Y).\end{equation}
If $M$ is generalized K\"ahler, then, in view of (\ref{relpsiOmega}), (\ref{FpmclasCRF}) becomes
\begin{equation}\label{FpmKclasCRF}\begin{array}{l} \iota^*(i(\nu)d\psi)(F_\pm X,F_\pm Y)=\iota^*(i(\nu)d\psi)(X,Y),\vspace*{2mm}\\
b(F_\pm X,F_\pm Y)=b(X,Y),\end{array}\end{equation} where $X,Y\in P_\pm$.

Now, we shall prove the following:
\begin{prop}\label{answer3} If $\iota:N\subseteq M$ is an oriented hypersurface of a generalized K\"ahler manifold, the induced generalized F structure of $N$ is CRFK iff one has
\begin{equation}\label{eqptans3} \begin{array}{l}
 \iota^*(i(\nu)d\psi)(F_\pm X,F_\pm Y)=\iota^*(i(\nu)d\psi)(X,Y),\;\forall X,Y\in P_\pm,\vspace*{2mm}\\
b(X,F_\pm U)=\mp\frac{1}{2}\iota^*(i(\nu)d\psi)(X,F_\pm U),\;\forall X\in TN,U\in P_\pm.
\end{array}\end{equation}\end{prop}
\begin{proof} If we assume that the induced structure is generalized CRFK, the first condition (\ref{eqptans3}) comes from (\ref{FpmKclasCRF}).

Then, let us look at the CRFK condition (2), Section 2. Using the general formula (\ref{legLieFJ}), we get
\begin{equation}\label{auxptgenK}\begin{array}{l} s(\nabla^s_XF_\pm(Y),F_\pm U)
=\gamma(\nabla^\gamma_XJ_\pm(Y),J_\pm U)\vspace*{2mm}\\ -\xi_\pm(U)\gamma(\nabla^\gamma_XJ_\pm(Y),\nu) +\xi_\pm(Y)b(X,F_\pm U),\;\forall X,Y,U\in TN.\end{array}\end{equation}
If we calculate the right hand side of (\ref{auxptgenK}) using the generalized K\"ahler condition (\ref{relpsiJ}) and require the result to be equal to the sign opposite of the right hand side of (\ref{CRFK6}) with $s,\kappa$ instead of $\gamma,\psi$, then, after reductions, the CRFK condition (\ref{CRFK6}) of the induced structure of the hypersurface turns out to be
$$\xi_\pm(Y)b(X,F_\pm U)=\pm\frac{1}{2}\xi_\pm(Y)d\psi(X,\nu,J_\pm U),\;\forall X,Y,U\in TN.$$
This condition identically holds for $Y\in P_\pm$ ($\xi_\pm(Y)=0$), hence, it suffices to ask it for $Y=Z_\pm$ ($\xi_\pm(Z_\pm)=1$). In this case, the result is an identity for $U=Z_\pm$ and it is the second condition (\ref{eqptans3}) for $U\in P_\pm$.

Therefore, if the induced structure is CRFK, the two conditions (\ref{eqptans3}) hold.

To prove the converse, since (\ref{eqptans3}) implies the first condition (\ref{FpmKclasCRF}) and (\ref{CRFK6}) (as shown by the calculation above), we only have to prove that (\ref{eqptans3}) also implies the second condition (\ref{FpmKclasCRF}). This follows by using $X\in P_\pm$ in the second equality (\ref{eqptans3}). (Of course, since $b$ is symmetric, $b(F_\pm X,U)$ is equal to the value given by (\ref{eqptans3}) for $b(U,F_\pm X)$.)
\end{proof}
\begin{corol}\label{Fpmnecnormal} If the induced structure of the hypersurface $N$ of a generalized K\"ahler manifold is CRFK, then, the almost contact structures $(F_\pm,Z_\pm,\xi_\pm)$ are normal.
\end{corol}
\begin{proof} Compute $b(Z_\pm,X)=-b(Z_\pm,F_\pm(F_\pm X))$ for $X\in P_\pm$ using (\ref{eqptans3}) and (\ref{relpsiOmega}). The result is (\ref{eqnormal2}).
\end{proof}
\begin{corol}\label{corolCRFK} Assume that the oriented hypersurface $N$ of the generalized K\"ahler manifold $M$ satisfies the condition $\iota^*(i(\nu)d\psi)=0$ (in particular, that $d\psi=0$). Then: (a) if $N$ is a totally geodesic submanifold, the induced structure of $N$ is CRFK, (b) if the induced structure is CRFK and either $\nabla^\gamma_{Z_+}Z_+=0$ or $\nabla^\gamma_{Z_-}Z_-=0$, the hypersurface $N$ is totally geodesic.\end{corol}
\begin{proof} Conclusion (a) is an obvious consequence of Proposition \ref{answer3}. For (b), firstly, the second condition (\ref{eqptans3}) yields the vanishing of the second fundamental form $b$ on all arguments, except for $b(Z_\pm,Z_\pm)$. Secondly, each of the additional hypotheses means that the trajectories of the corresponding vector field $Z_\pm$ are geodesics of $\gamma$, whence $b(Z_\pm,Z_\pm)$ vanishes as well.
\end{proof}
\section{$(2,1)$-generalized almost contact structures}
In this section we begin by noticing some properties of the induced generalized F structure of a hypersurface of an arbitrary generalized almost Hermitian manifold. Then, we use these properties to define new structures that naturally extend the classical almost contact and metric almost contact structures to the generalized framework.

From Section 2, we recall that, if $\iota:N\rightarrow M$ is an oriented hypersurface of the generalized almost Hermitian manifold $(M,\gamma,\psi,J_\pm)$ and $(F_\pm,Z_\pm,\xi_\pm,s=\iota^*\gamma,\kappa=\iota^*\psi)$ are the induced structures of $N$, then $\mathcal{F}\Leftrightarrow(F_\pm),\,G'\Leftrightarrow(s,\kappa)$ define a generalized metric F structure. Furthermore, we may define the cross sections
$$\mathcal{Z}_\pm=(Z_\pm,\flat_{\kappa\pm s}Z_\pm)\in V'_\pm\subseteq\Gamma\mathbf{T}N,$$
where $V'_\pm$ are the bundles (\ref{exprEpm}) defined by the metric $G'$.

For this induced structure we can prove
\begin{prop}\label{newprop} The structure $(\mathcal{F},G')$ and the cross sections	$\mathcal{Z}_\pm$ satisfy the following properties
\begin{equation}\label{almoctZpm}
g(\mathcal{Z}_+,\mathcal{Z}_-)=0,\,g(\mathcal{Z}_+,\mathcal{Z}_+)=1,\, g(\mathcal{Z}_-,\mathcal{Z}_-)=-1,\end{equation}
\begin{equation}\label{almctF2} \mathcal{F}\mathcal{Z}_\pm=0,\;\mathcal{F}^2=-Id+(\flat_g\mathcal{Z}_+)
\otimes\mathcal{Z}_+ - (\flat_g\mathcal{Z}_-)\otimes\mathcal{Z}_-,\end{equation}
\begin{equation}\label{21metriccuZpm} G'(\mathcal{F}\mathcal{X},\mathcal{F}\mathcal{Y})
=G'(\mathcal{X},\mathcal{Y})-g(\mathcal{Z}_+,\mathcal{X})g(\mathcal{Z}_+,\mathcal{Y})
+g(\mathcal{Z}_-,\mathcal{X})g(\mathcal{Z}_-,\mathcal{Y}),\end{equation}
where $g$ is the pairing metric on $\mathbf{T}N$.\end{prop}
\begin{proof}
The equalities (\ref{almoctZpm}) follow from $V'_+\perp_g V'_-$, $g|_{V'_\pm}=\pm s$ and $||Z_\pm||_s=1$. The equalities (\ref{almctF2}) follow	 by using (\ref{eqJrond}) for $\mathcal{F}$ and $F_\pm$.
Particularly, on $\mathcal{X}\in V'_\pm$, the second equality (\ref{almctF2}) follows from
$$\mathcal{F}^2(X,\flat_{\kappa\pm s}X)\stackrel{(\ref{almcont})}{=}-(X,\flat_{\kappa\pm s}X) +\xi_\pm(X)(Z_\pm,\flat_{\kappa\pm s}Z_\pm)$$ and
$$\begin{array}{lcl}\flat_g\mathcal{Z}_\pm((X,\flat_{\kappa\pm s}X))&=&g((Z_\pm,\flat_{\kappa\pm s}Z_\pm),(X,\flat_{\kappa\pm s}X))\vspace*{2mm}\\ &=&\pm s(Z_\pm,X)=\pm\xi_\pm(X).\end{array}$$
Finally, (\ref{21metriccuZpm}) follows from the classical metric property (\ref{clasmetric}).
\end{proof}

From \cite{V1}, we recall that a set of $g$-orthogonal cross sections $\mathcal{Z}_a,\mathcal{Z}_\alpha\in\Gamma\mathbf{T}M$ is a {\it complementary frame} of the generalized F structure $\mathcal{F}$ if $g(\mathcal{Z}_a,\mathcal{Z}_a)=-1, g(\mathcal{Z}_\alpha,\mathcal{Z}_\alpha)=1$, $\forall a=1,...,neg\,\mathcal{F},\alpha=1,...,corank\,\mathcal{F}-neg\,\mathcal{F}$ and
\begin{equation}\label{comfr} \mathcal{F}\mathcal{Z}_a=0,\,\mathcal{F}\mathcal{Z}_\alpha=0,\,
\mathcal{F}^2=-Id+
\sum_\alpha(\flat_g\mathcal{Z}_\alpha)\otimes\mathcal{Z}_\alpha- \sum_a(\flat_g\mathcal{Z}_a)\otimes\mathcal{Z}_a.\end{equation}
Any generalized F structure has local complementary frames and, if such frames exist globally we say that the structure {\it has complementary frames}. In fact, a generalized F structure with complementary frames may be defined directly as a system $(\mathcal{F},\mathcal{Z}_a,\mathcal{Z}_\alpha)$, where $\mathcal{F}$ is a $g$-skew-symmetric endomorphism of $\mathbf{T}M$ and all the above conditions are satisfied. Indeed, then, (\ref{comfr}) implies $\mathcal{F}^3+\mathcal{F}=0$. It also implies that $S=ker\,\mathcal{F}=span\{\mathcal{Z}_a,\mathcal{Z}_\alpha\}$ and $\mathbf{T}M=L\oplus S$ ($L=im\,\mathcal{F}$). The last condition (\ref{comfr}) yields
\begin{equation}\label{prScuframe} pr_S\mathcal{X}=\mathcal{X}+\mathcal{F}^2\mathcal{X} =\sum_\alpha g(\mathcal{Z}_\alpha,\mathcal{X})\mathcal{Z}_\alpha- g(\mathcal{Z}_a,\mathcal{X})\mathcal{Z}_a, \;\forall\mathcal{X}\in\mathbf{T}M.\end{equation}

Thus, (\ref{almoctZpm}) and (\ref{almctF2}) show that the induced generalized F structure of a hypersurface of a generalized almost Hermitian manifold has complementary frames, corank $2$ and negative index $1$.

The interesting point is that the same holds for a classical almost contact structure $(F,Z,\xi)$ of an arbitrary manifold $M$. Indeed, then, the equivalent generalized F structure is $\mathcal{F}(X,\alpha)=(FX,-\alpha\circ F)$ and, if we define $\mathcal{Z}_+=(Z,\xi),\mathcal{Z}_-=(Z,-\xi)$, conditions (\ref{almoctZpm}) and (\ref{almctF2}) are satisfied. We shall refer to this interpretation of $(F,Z,\xi)$ by saying that ``we see $(F,Z,\xi)$ as a $(2,1)$-generalized structure".

In view of the last remark it is natural to give the following definition.
\begin{defin}\label{21genct} {\rm  On any manifold $M$, a {\it $(2,1)$-generalized almost contact structure} is a generalized F structure $\mathcal{F}$ of corank $2$ and negative index $1$, together with a fixed complementary frame $\mathcal{Z}_-,\mathcal{Z}_+$. Furthermore, if we add a generalized Riemannian metric $G\Leftrightarrow(\gamma,\psi)$ such that (\ref{21metriccuZpm}), with $s\mapsto\gamma,\kappa\mapsto\psi$, also holds, $(\mathcal{F},G)$ is a {\it metric $(2,1)$-generalized almost contact structure}.}\end{defin}

We added the characteristics $(2,1)$ to the name since other notions of generalized almost contact structures already exist in the literature \cite{{IW},{PW},{V3}}. The $(2,1)$-generalized almost contact structures were also introduced under a different form in \cite{Sek}, where the author looks at $\mathcal{E}_\pm=(1/2)(\mathcal{Z}_+\pm\mathcal{Z}_-)$ rather than at $\mathcal{Z}_\pm$. Notice that, if $(\mathcal{F},\mathcal{Z}_\pm)$ is a $(2,1)$-generalized almost contact structure, $L=im\,\mathcal{F}$ is a split structure in the sense of \cite{{AD1},{AD2}}.

We may extend the classical metric compatibility condition (\ref{clasmetric}) to the case of generalized F structures $\mathcal{F}$ with complementary frame and define compatibility with a generalized Riemannian metric $G\Leftrightarrow(\gamma,\psi)$ by the condition
\begin{equation}\label{metriccomplfr} \begin{array}{lcl}
G(\mathcal{F}\mathcal{X},\mathcal{F}\mathcal{Y}) =G(\mathcal{X},\mathcal{Y})&-&\sum_\alpha g(\mathcal{Z}_\alpha,\mathcal{X})g(\mathcal{Z}_\alpha,\mathcal{Y})\vspace*{2mm}\\ &+&\sum_a g(\mathcal{Z}_a,\mathcal{X})g(\mathcal{Z}_a,\mathcal{Y}).\end{array}\end{equation}

This condition implies (\ref{G-F}). Moreover, using $\mathcal{Z}_\alpha,\mathcal{Z}_a$ as arguments in (\ref{metriccomplfr}), we get $G(\mathcal{Z}_\alpha,\mathcal{F}\mathcal{X})=
G(\mathcal{Z}_a,\mathcal{F}\mathcal{X})=0$, $G(\mathcal{Z}_\alpha,\mathcal{Z}_\beta) =\delta_{\alpha,\beta}$, $G(\mathcal{Z}_a,\mathcal{Z}_b)=\delta_{ab}$, $G(\mathcal{Z}_\alpha,\mathcal{Z}_a)=0$, whence, for $\mathcal{G}$ defined by (\ref{condptGrond}), we get $\mathcal{G}\mathcal{Z}_\alpha=\mathcal{Z}_\alpha,\,
\mathcal{G}\mathcal{Z}_a=-\mathcal{Z}_a$, therefore, $\mathcal{Z}_\alpha\in V_+,Z_a\in V_-$.
Thus, we may put $\mathcal{Z}_\alpha=(Z_\alpha,\flat_{\psi+\gamma}Z_\alpha)$, $\mathcal{Z}_a= (Z_a,\flat_{\psi-\gamma}Z_a)$, where $\gamma(Z_\alpha,Z_\beta)=\delta_{\alpha\beta},
\gamma(Z_a,Z_b)=\delta_{ab}$. But,	$\gamma(Z_\alpha,Z_a)\neq0$ is possible.

Another interpretation of the conditions discussed above is obtained with the $1$-forms $\xi_\alpha=\flat_\gamma Z_\alpha,\,\xi_a=\flat_\gamma Z_a$. Since $g|_{V_\pm}=\pm \gamma\circ\tau_\pm$, it follows that (\ref{comfr}) is equivalent to
$$\begin{array}{l} F_+(Z_\alpha)=0,\;	 F_+^2=-Id+\sum_\alpha\xi_\alpha\otimes Z_\alpha,\vspace*{2mm}\\	 F_-(Z_a)=0,\; F_-^2=-Id+\sum_a\xi_a\otimes Z_a, \end{array}$$ and (\ref{metriccomplfr}) is equivalent to
$$\begin{array}{l} \gamma(F_+X,F_+Y)=\gamma(X,Y)-\sum_\alpha\xi_\alpha(X)\xi_\alpha(Y), \vspace*{2mm}\\ \gamma(F_-X,F_-Y)=\gamma(X,Y)-\sum_a\xi_a(X)\xi_a(Y),\end{array}$$
where $F_\pm$ are the classical F structures that define $\mathcal{F}$.

In particular, we proved
\begin{prop}\label{prop21metric}
A metric $(2,1)$-generalized almost contact structure is equivalent to a pair of classical metric almost contact structures $(F_\pm,Z_\pm,\xi_\pm,\gamma)$ that have the same metric $\gamma$, accompanied by a $2$-form $\psi$.\end{prop}
\begin{corol}\label{corol21indus} 1. The induced generalized structure of an orientable hypersurface of a generalized almost Hermitian manifold is metric. 2. A classical, metric, almost contact structure $(F,Z,\xi,\gamma)$ seen as a generalized (2,1)-almost contact structure and $G$ defined by $(\gamma,\psi=0)$ satisfy the metric compatibility (\ref{21metriccuZpm}).\end{corol}
\begin{proof} 1. This follows from the equality (\ref{21metriccuZpm}) as well as from the corresponding pair of classical structures. 2. With $Z_+=Z_-=Z$, we get a corresponding classical pair. In this case, $\gamma(Z_+,Z_-)=1\neq0$.\end{proof}

For any $(2,1)$-generalized almost contact structure $\mathcal{F}$ there exists a naturally associated generalized almost complex structure $\mathcal{J}$ on the manifold $M\times\mathds{R}$. Indeed, for the latter we have
\begin{equation}\label{desctxR}\mathbf{T}(M\times\mathds{R})
=im\,\mathcal{F}\oplus_{\perp_g} span\{\mathcal{Z}_+,\mathcal{Z}_-\} \oplus_{\perp_g} span\{\mathcal{T}_+,\mathcal{T}_-\},
\end{equation}
where $g$ is the pairing metric of $M\times\mathds{R}$ and $\mathcal{T}_+=(\partial/\partial t,dt),\mathcal{T}_-=(-\partial/\partial t,dt)$, $t$ being the coordinate on $\mathds{R}$.
On the sum of the last two terms of (\ref{desctxR}), we have the complex structure
\begin{equation}\label{fKrond} \mathcal{K}(\mathcal{T}_+)=-\mathcal{Z}_+, \,\mathcal{K}(\mathcal{T}_-)=\mathcal{Z}_-,\;\mathcal{K}^2=-Id,\end{equation}
which extends to $\mathcal{K}\in End(\mathbf{T}(M\times\mathds{R}))$ by asking $\mathcal{K}\circ\mathcal{F}=0$.
If we also extend $\mathcal{F}$ by $\mathcal{F}\circ\mathcal{K}=0$, then, $\mathcal{J}=\mathcal{F}+\mathcal{K}$ is the announced generalized almost complex structure of $M\times\mathds{R}$. (The $g$-skew symmetry of $\mathcal{J}$ is easily checked on the basis $\mathcal{Z}_\pm,\mathcal{T}_\pm$.) Explicitly, we have
\begin{equation}\label{JptFrond}\begin{array}{lcl} \mathcal{J}&=&\mathcal{F} +(\flat_g\mathcal{Z}_+)\otimes\mathcal{T}_+ +(\flat_g\mathcal{Z}_-)\otimes\mathcal{T}_- \vspace*{2mm}\\
&-&(\flat_g\mathcal{T}_+)\otimes\mathcal{Z}_+
-(\flat_g\mathcal{T}_-)\otimes\mathcal{Z}_-.\end{array}\end{equation}
\begin{rem}\label{dimens21} {\rm (1) In fact, there exists a one-parameter family of such structures since we may use the vectors of any $g$-pseudo-orthonormal basis of $\mathbf{T}\mathds{R}\approx\mathds{R}^2$ instead of $\mathcal{T}_\pm$ chosen above. (2) The existence of $\mathcal{J}$ implies the even dimensionality of $M\times\mathds{R}$. Hence, a $(2,1)$-generalized almost contact structure may exist only on odd-dimensional manifolds $M$.
(3) For the classical almost contact structure $(F,Z,\xi)$ seen as a $(2,1)$-generalized structure, formula (\ref{JptFrond}) yields the generalized almost complex structure defined by the classical structure (\ref{JF}).}\end{rem}

Inspired by \cite{GY}, we notice one more interesting structure that exists on a $(2,1)$-generalized almost contact manifold, namely, the complex endomorphism $\Phi\in End(\mathbf{T}M)$ defined by the formula
\begin{equation}\label{eqPhi} \Phi=i\mathcal{F}+(\flat_g\mathcal{Z}_+)\otimes\mathcal{Z}_- -(\flat_g\mathcal{Z}_-)\otimes\mathcal{Z}_+.\end{equation}
$\Phi$ has the following properties
$$ g(\Phi\mathcal{X},\mathcal{Y})+g(\mathcal{X},\Phi\mathcal{Y})=0,\; \Phi^2=Id,$$
hence, it may be seen as a complex analogue of a generalized almost product structure \cite{V}.
The structure $\Phi$ is equivalent with the decomposition
\begin{equation}\label{eqGY}
\mathbf{T}_cM=[\bar{E}\oplus span\{\mathcal{Z}_++\mathcal{Z}_-\}] \oplus[E\oplus span\{\mathcal{Z}_+-\mathcal{Z}_-\}],\end{equation}
where the square  brackets are the $\pm1$-eigenbundles of $\Phi$.

In the terminology of \cite{{PW},{Sek}}, one of the terms of (\ref{eqGY}) is closed under Courant brackets iff we are in the case of a generalized contact structure and both terms are closed iff we have a strong generalized contact structure. Using eigenvectors of $\Phi$ as arguments it is easy to check that the structure is generalized contact iff one of the equalities $\mathcal{N}_\Phi\pm\mathcal{N}_{\bar{\Phi}}=0$ holds and it is strong generalized contact iff $\mathcal{N}_{\Phi}=0$.

Furthermore, if $\mathcal{F}$ is metric with respect to $G$, i.e., (\ref{metriccomplfr}) holds, $\Phi$ has the following property
\begin{equation}\label{PhiG}
G(\Phi\mathcal{X},\Phi\mathcal{Y})=-G(\mathcal{X},\mathcal{Y}).\end{equation}
\begin{example}\label{exptPhiclasic} {\rm
In the case of a classical almost contact structure, seen as a $(2,1)$-generalized structure, formula (\ref{eqPhi}) yields
\begin{equation}\label{Phiptclasic} \Phi(X,\alpha)= (\phi X,-\alpha\circ\phi),\end{equation}
where $\phi\in End(T^cM)$ is defined by
$$\phi X=iFX+\xi(X)Z.$$
The endomorphism $\phi$ has the property $\phi^2=Id$ and has the $\mp 1$-eigenbundles $H,\bar{H}\oplus Q_c$, respectively, where $H,\bar{H}$ are the $\pm i$-eigenbundles of $F$ and $Q=span\,\{Z\}$. As a consequence, it follows that a classical CRF almost contact structure is characterized by $\mathcal{N}_\phi=0$.}\end{example}

If the classical structure is compatible with the Riemannian metric $\gamma$ and $G\Leftrightarrow(\gamma,\psi)$, the equality (\ref{PhiG}) reduces to $\gamma(\phi Y,\phi Y)=-\gamma(X,Y)$. This is an interesting remark since it shows that a metric, almost contact structure produces generalized almost para-Hermitian structures given by the triples $(\gamma,\psi,\phi)$. (See \cite{V5}, Proposition 2.2.)
Furthermore, if we use a closed form $\psi$, the corresponding structure is generalized para-K\"ahler iff the initially given classical metric almost contact structure is cosymplectic in the sense of Blair. Indeed, in Example \ref{CRFKclasic} we saw that a cosymplectic manifold is classical CRF, hence, $\mathcal{N}_\phi=0$ and the eigenbundles of $\phi$ indicated above are involutive. Then, $\nabla^\gamma F=0,\nabla^\gamma Z=0$ imply $\nabla^\gamma \phi=0$. Thus, the integrability conditions required by Theorem 3.1 of \cite{V5} hold. Conversely, $\nabla^\gamma \phi=0$ implies $\nabla^\gamma F=0$.
\section{Normal $(2,1)$-generalized structures}
In this section we study the normality of $(2,1)$-generalized almost contact structures.
\begin{defin}\label{normalgen} {\rm A $(2,1)$-generalized almost contact structure on $M$ is {\it normal} if the associated generalized almost complex structure $\mathcal{J}$ defined by (\ref{JptFrond}) on $M\times\mathds{R}$ is integrable.}\end{defin}
\begin{rem}\label{defimplicaclasic} {\rm In view of Remark \ref{dimens21} (3), the $(2,1)$-generalized almost contact structure defined by a classical almost contact structure $(F,Z,\xi)$ is normal in the sense of Definition \ref{normalgen} iff it is normal in the classical sense.}\end{rem}
\begin{prop}\label{normalFrond}
The $(2,1)$-generalized almost contact structure $(\mathcal{F},\mathcal{Z}_\pm)$ with associated generalized almost complex structure $\mathcal{J}$ is normal iff the following conditions hold
\begin{equation}\label{normaltotal}
[\mathcal{Z}_+,\mathcal{Z}_-]=0,\; [\mathcal{Z}_\pm,\mathcal{F}\mathcal{X}]=\mathcal{F}[\mathcal{Z}_\pm,\mathcal{X}],\;
\mathcal{N}_{\mathcal{F}}(\mathcal{X},\mathcal{Y})=pr_S[\mathcal{X},\mathcal{Y}],
\end{equation}
where $S=ker\,\mathcal{F}$ and $\mathcal{X},\mathcal{Y}\in L=im\,\mathcal{F}$.
\end{prop}
\begin{proof}
Before we start, let us notice that	 the condition $\mathcal{X},\mathcal{Y}\in L$ ensures the invariance of conditions (\ref{normaltotal}) under the multiplication of these arguments by a function (use (\ref{propofC})).

We shall write down the integrability condition $\mathcal{N}_{\mathcal{J}}=0$ for all the basic pairs of arguments: (i) $(\mathcal{T}_+,\mathcal{T}_-)$, (ii)
$(\mathcal{T}_\pm,\mathcal{Z}_\pm)$, (iii) $(\mathcal{Z}_+,\mathcal{Z}_-)$, (iv) $(\mathcal{X}\in L,\mathcal{Z}_\pm)$, (v) $(\mathcal{X}\in L,\mathcal{T}_\pm)$, (vi) $(\mathcal{X}\in L,\mathcal{Y}\in L)$.

The three cases (i), (ii), (iii) yield a single condition, namely
$[\mathcal{Z}_+,\mathcal{Z}_-]=0$. In case (iv), since $\mathcal{J}=\mathcal{F}+\mathcal{K}$ and $\mathcal{J}\mathcal{X}=\mathcal{F}\mathcal{X}$ for $\mathcal{X}\in L$, the condition $\mathcal{N}_{\mathcal{J}}(\mathcal{X},\mathcal{Z}_\pm)=0$ becomes
\begin{equation}\label{integr4} [\mathcal{X},\mathcal{Z}_\pm]+ \mathcal{F}[\mathcal{F}\mathcal{X},\mathcal{Z}_\pm]= -\mathcal{K}[\mathcal{F}\mathcal{X},\mathcal{Z}_\pm].\end{equation}

Obviously, $\mathcal{K}(\mathbf{T}M)\subseteq span\{\mathcal{T}_\pm\}$. Accordingly, since the Courant brackets in (\ref{integr4}) belong to $\Gamma\mathbf{T}M$, (\ref{integr4}) is equivalent to the vanishing of its two sides. But, $\forall\mathcal{U}\in\mathbf{T}M$, $\mathcal{K}\mathcal{U}=\mathcal{K}(pr_S\mathcal{U})$, hence, $\mathcal{K}\mathcal{U}=0$ is equivalent to $\mathcal{U}\in L$, therefore, (\ref{integr4}) means
\begin{equation}\label{integr43} [\mathcal{X},\mathcal{Z}_\pm]+\mathcal{F}[\mathcal{F}\mathcal{X},\mathcal{Z}_\pm,]=0,
\;[\mathcal{F}\mathcal{X},\mathcal{Z}_\pm]\in L.\end{equation}
In (\ref{integr43}) we may replace $\mathcal{X}\in L$ by $\mathcal{F}\mathcal{X}$, with the consequence $\mathcal{F}^2\mathcal{X}=-\mathcal{X}$. Then, the first condition (\ref{integr43}) becomes the second condition (\ref{normaltotal}) and it implies the second condition, which may also be seen as \begin{equation}\label{integr42}\mathcal{X}\in L\,\Rightarrow\,[\mathcal{Z}_\pm,\mathcal{X}]\in L.\end{equation}

In case (v) one gets the same result as in case (iv).

Finally, in case (vi), we have $\mathcal{J}\mathcal{X}=\mathcal{F}\mathcal{X},\mathcal{J}\mathcal{Y}=\mathcal{F}\mathcal{Y}$
and $\mathcal{N}_\mathcal{F}(\mathcal{X},\mathcal{Y})=0$ becomes (use also (\ref{prScuframe}))
\begin{equation}\label{integr6} \mathcal{N}_{\mathcal{F}}(\mathcal{X},\mathcal{Y}) -pr_S[\mathcal{X},\mathcal{Y}] =\mathcal{K}([\mathcal{F}\mathcal{X},\mathcal{Y}] +[\mathcal{X},\mathcal{F}\mathcal{Y}]).\end{equation}
As in the case of (\ref{integr4}), (\ref{integr6}) is equivalent to
\begin{equation}\label{ptobsCRF}
\mathcal{N}_{\mathcal{F}}(\mathcal{X},\mathcal{Y})=pr_S[\mathcal{X},\mathcal{Y}], \;[\mathcal{F}\mathcal{X},\mathcal{Y}]+[\mathcal{X},\mathcal{F}\mathcal{Y}]\in L.\end{equation}
The first condition is the last condition (\ref{normaltotal}). As we know, this is the integrability condition of the generalized F structure $\mathcal{F}$ and it ensures the closure of the $i$-eigenbundle $E$ of $\mathcal{F}$ under Courant brackets. This allows us to see that the second condition (\ref{ptobsCRF}) is implied by looking at the three possible cases of arguments $\mathcal{X},\mathcal{Y}$ in $E,\bar{E}$.

Therefore, the normality of the structure is characterized by the conditions (\ref{normaltotal}).
\end{proof}
\begin{rem}\label{obsparametru} {\rm The proof of Proposition \ref{normalFrond} is valid if $\mathcal{T}_\pm$ are replaced by any $g$-pseudo-orthonormal basis of $\mathbf{T}\mathds{R}$, hence, if $\mathcal{F}$ is normal, $M\times\mathds{R}$ has a one-parameter family of generalized complex structures.}\end{rem}
\begin{rem}\label{obsgennorm} {\rm As in the classical case, the three conditions (\ref{normaltotal}) may be unified into a single expression. To do that, we use the {\it naive exterior differential} of a Courant algebroid, which leads to the {\it naive cohomology} \cite{SX}. In particular, for $\mathcal{U}\in\Gamma\mathbf{T}M$, we define $d_C\mathcal{U}:\Gamma\wedge^2\mathbf{T}M
\rightarrow C^\infty(M)$ by
$$d_C\mathcal{U}(\mathcal{X},\mathcal{Y})=pr_{TM}\mathcal{X}(g(\mathcal{U},\mathcal{Y}))
-pr_{TM}\mathcal{Y}(g(\mathcal{U},\mathcal{X}))-g(\mathcal{U},[\mathcal{X},\mathcal{Y}]).$$
This operator is not bilinear over $C^\infty(M)$.

However, using the basic type arguments as in the proof of Proposition \ref{normalFrond}, it is easy to see that (\ref{normaltotal}) is equivalent to
\begin{equation}\label{normtotal2}
\mathcal{N}_{\mathcal{F}}(\mathcal{X},\mathcal{Y})+ d_C\mathcal{Z}_+(\mathcal{X},\mathcal{Y})\mathcal{Z}_+ -d_C\mathcal{Z}_-(\mathcal{X},\mathcal{Y})\mathcal{Z}_-=0\; (\forall\mathcal{X},\mathcal{Y}\in\Gamma\mathbf{T}M)\end{equation}
and using (\ref{propofC}) we see
that the left hand side of (\ref{normtotal2}) is $C^\infty(M)$-bilinear.}\end{rem}
\begin{rem}\label{normdaCRF} {\rm During the proof, we noticed that normality implies integrability of $\mathcal{F}$, in other words, in the case of a normal structure, $\mathcal{F}$ is a CRF structure.}\end{rem}

We will also prove the following proposition.
\begin{prop}\label{normaldaPhi} If $(\mathcal{F},\mathcal{Z}_\pm)$ is a normal $(2,1)$-generalized almost contact structure, the endomorphism $\Phi$ defined by (\ref{eqPhi}) satisfies the integrability condition $\mathcal{N}_\Phi=0$.\end{prop}
\begin{proof}
In view of (\ref{prScuframe}), we have
$\Phi\mathcal{X}=i\mathcal{F}\mathcal{X}+pr_S\mathcal{X}$.
We shall compute $\mathcal{N}_\Phi$ for pairs of arguments (a) $\mathcal{X},\mathcal{Y}\in L$, (b) $\mathcal{X}\in L,\mathcal{Z}_\pm$, (c) $\mathcal{Z}_+,\mathcal{Z}_-$.
In case (c) $\mathcal{N}_\Phi(\mathcal{Z}_+,\mathcal{Z}_-)=0$ because of the first condition (\ref{normaltotal}).
In case (a), $\Phi\mathcal{X}=i\mathcal{F}\mathcal{X},\Phi\mathcal{Y}=i\mathcal{F}\mathcal{Y}$ and we get
\begin{equation}\label{nPhipeLa} \mathcal{N}_\Phi(\mathcal{X},\mathcal{Y})=
-[\mathcal{F}\mathcal{X},\mathcal{F}\mathcal{Y}] -i\Phi([\mathcal{F}\mathcal{X},\mathcal{Y}]
+[\mathcal{X},\mathcal{F}\mathcal{Y}])+[\mathcal{X},\mathcal{Y}].
\end{equation}
Since the structure is normal, the second condition (\ref{ptobsCRF}) allows us to replace $\Phi$ by $i\mathcal{F}$ in the right hand side of (\ref{nPhipeLa}), which (modulo the last condition (\ref{normaltotal})) leads to
$$\mathcal{N}_\Phi(\mathcal{X},\mathcal{Y})= -\mathcal{N}_\mathcal{F}(\mathcal{X},\mathcal{Y})+pr_S[\mathcal{X},\mathcal{Y}]=0.$$
In case (b) and, again in view of normality, we may use (\ref{integr42}) and we get
$$\mathcal{N}_\Phi(\mathcal{X},\mathcal{Z}_\pm)= i([\mathcal{F}\mathcal{X},\mathcal{Z}_\pm]-\mathcal{F}[\mathcal{X},\mathcal{Z}_\pm])
+([\mathcal{X},\mathcal{Z}_\pm]+\mathcal{F}[\mathcal{F}\mathcal{X},\mathcal{Z}_\pm])=0,
$$
where the last equality sign is justified by the second condition (\ref{normalFrond}) and (\ref{integr43}).
\end{proof}

Proposition \ref{normaldaPhi} shows that the normal structures are strong generalized contact structures, whence, we see that Proposition \ref{normalFrond} is equivalent to Proposition 4.2 of \cite{Sek}.

Among others, the interest of Proposition \ref{normaldaPhi} comes from the fact that the condition $\mathcal{N}_\Phi=0$ may be treated in exactly the same way as the integrability condition of a generalized almost product structure, except for the fact that everything will be over the complex field $\mathds{C}$ instead of $\mathds{R}$. In particular, we may consider a matrix representation of $\Phi$ via complex tensors $A\in End(T_cM),\sigma\in\Omega^2(M,\mathds{C}),\pi\in\wedge^2T_cM$ and, in the integrable case, $\pi$ is a complex Poisson bivector field \cite{V}. Formula (\ref{Phiptclasic}) shows that the Poisson structure of a classical almost contact structure is zero.

Now, we consider normality in the metric case.
Let $(\mathcal{F},\mathcal{Z}_\pm,G)$ be a metric $(2,1)$-generalized almost contact structure, $\mathcal{F}'$ its associated second structure (see Section 2) and $\mathcal{Z}'_\pm=\pm\mathcal{Z}_\pm$. Using, (\ref{condptGrond}), it is easy to see that $\mathcal{G}$ commutes with $\mathcal{F}$, $\mathcal{F}'=\mathcal{G}\circ\mathcal{F}$ and $\mathcal{F}'\circ\mathcal{F}=\mathcal{F}\circ\mathcal{F}'$. Accordingly, $\mathcal{F}'\mathcal{Z}'_\pm=0,\,\mathcal{F}^{'2}=\mathcal{F}^2$ and it follows that $(\mathcal{F}',\mathcal{Z}'_\pm,G)$ is again a metric $(2,1)$-generalized almost contact structure. \begin{defin}\label{binormal} {\rm The structure $(\mathcal{F},\mathcal{Z}_\pm,G)$ is {\it binormal} if it is normal and the structure $(\mathcal{F}',\mathcal{Z}'_\pm,G)$ is also normal.}\end{defin}
\begin{rem}\label{obsGT} {\rm
If we put $\mathcal{E}_\pm=(1/2)(\mathcal{Z}_+\pm\mathcal{Z}_-)$, we see that the binormal structures coincide with the generalized coK\"ahler structures of \cite{GT}.}\end{rem}
\begin{prop}\label{propbinorm} The structure $(\mathcal{F},\mathcal{Z}_\pm,G)$ is a binormal structure iff the generalized almost complex structures $\mathcal{J},\mathcal{J}'$ produced by $\mathcal{F},\mathcal{F}'$ on $M\times\mathds{R}$ define a generalized K\"ahler structure.\end{prop}
\begin{proof} If $(\mathcal{J},\mathcal{J}')$ define a generalized K\"ahler structure, then, $\mathcal{J}$ and $\mathcal{J}'$ are integrable, hence, $(\mathcal{F},\mathcal{Z}_\pm,G)$ is a binormal structure.

Conversely, from the commutation of $\mathcal{F},\mathcal{F}'$ we deduce $im\,\mathcal{F}=im\,\mathcal{F}'=L$ and the restrictions of $\mathcal{J},\mathcal{J}'$ to this common image are $\mathcal{F},\mathcal{F}'$. Therefore, since $\mathcal{F},\mathcal{F}'$ commute, $\mathcal{J},\mathcal{J}'$commute on $L$. Then, the use of (\ref{fKrond}) for $\mathcal{Z}_\pm,\mathcal{Z}'_\pm$ shows that $\mathcal{J},\mathcal{J}'$ also commute on $im\,\mathcal{K}=im\,\mathcal{K}'$. Hence, $\mathcal{J},\mathcal{J}'$ commute on $\mathbf{T}(M\times\mathds{R})$. If $(\mathcal{F},\mathcal{Z}_\pm,G)$ is binormal, $\mathcal{J},\mathcal{J}'$ are integrable. Such a pair of generalized complex structures yields a generalized K\"ahler structure, provided that the metric $\tilde{G}$ defined on $M\times\mathds{R}$ by $\tilde{\mathcal{G}}=-\mathcal{J}\circ\mathcal{J}'$ is positive definite \cite{Gualt}. The last condition holds. Indeed, using the second relation (\ref{condptGrond}) to express $\tilde{G}$, we get $$\tilde{G}|_L=G|_L,\,\tilde{G}|_S=G|_S,\,\tilde{G}(\mathcal{T}_\pm,\mathcal{T}_\pm)=1,\, \tilde{G}(\mathcal{T}_+,\mathcal{T}_-)=0,$$
which justifies the conclusion.
\end{proof}
\begin{rem}\label{obsU} {\rm By Remarks \ref{dimens21} (1), \ref{obsparametru}, Proposition \ref{propbinorm} remains valid if we replace $\mathcal{T}_\pm=(\pm(\partial/\partial t),dt)$ by any $g$-pseudo-orthonormal basis $\mathcal{T}'_\pm$ of $\mathbf{T}\mathds{R}$ in the definition of $\mathcal{J},\mathcal{J}'$.}\end{rem}
\begin{prop}\label{binorningK} Let $M$ be a generalized K\"ahler manifold with the structure $(J_\pm,\gamma,\psi)$. Assume that there exists a unit Killing vector field $U$ on $(M,\gamma)$ that is holomorphic with respect to both structures $J_\pm$. Let $N$ be a hypersurface of $M$ that is orthogonal to $U$. Then, the induced generalized metric almost contact structure of $N$ is a binormal structure.\end{prop}
\begin{proof}
We denote by $\mathcal{J},\mathcal{J}'$ the integrable generalized complex structures of the given generalized K\"ahler structure of $M$. By Definition \ref{binormal}, we have to show that the generalized almost complex structures, say $\tilde{\mathcal{J}},\tilde{\mathcal{J}}'$ associated to the induced structures $\mathcal{F},\mathcal{F}'$ of $N$ are integrable.
Since the normality conditions (\ref{normaltotal}) are local conditions on $N$, we shall proceed as in the proof of Proposition \ref{prop1Killing}, from where we recall the existence of neighborhoods $\mathcal{V}_x,\mathcal{U}_x$, $\forall x\in N$, such that $\mathcal{V}_x=\mathcal{U}_x\times(-\epsilon_x,\epsilon_x)$ ($0<\epsilon_x\in\mathds{R}$) and
\begin{equation}\label{JpminbnormK2}J_{\pm}=F_{\pm}-dt\otimes Z_\pm+\xi_\pm\otimes\frac{\partial}{\partial t}\end{equation} holds on $\mathcal{V}_x$.

Furthermore, as in the proof of Proposition \ref{prop1Killing} and under the present hypotheses, the expression (\ref{JpminbnormK2}) implies that $\mathcal{J}|_{\mathcal{V}_x}= \tilde{\mathcal{J}}|_{\mathcal{V}_x},\,\mathcal{J}'|_{\mathcal{V}_x} =\tilde{\mathcal{J}}'|_{\mathcal{V}_x}$ and the integrability of $\mathcal{J}|_{\mathcal{V}_x},\mathcal{J}'|_{\mathcal{V}_x}$ yields the required integrability of
$\tilde{\mathcal{J}}|_{\mathcal{V}_x},\tilde{\mathcal{J}}'|_{\mathcal{V}_x}$.
\end{proof}
\begin{corol}\label{corolKilling}  Let $M$ be a compact generalized K\"ahler manifold with a closed $2$-form $\psi$. Assume that there exists a unit Killing vector field $U$ on $(M,\gamma)$. Let $N$ be a hypersurface of $M$ that is orthogonal to $U$. Then, the induced generalized metric almost contact structure of $N$ is a binormal structure.\end{corol}
\begin{proof} In the indicated case, $(\gamma,J_\pm)$ are K\"ahler structures and, since $M$ is compact, the Killing field $U$ is holomorphic with respect to both $J_\pm$ \cite{Mor}.\end{proof}

Finally, we also give the expression of the normality and binormality properties by means of the pair of classical structures that is equivalent to the given $(2,1)$-generalized metric almost contact structure.
\begin{prop}\label{metricnormexplicit}	Let $(\mathcal{F},\mathcal{Z}_\pm,G)$ be a metric $(2,1)$-generalized almost contact structure and let $(F_\pm,Z_\pm,\xi_\pm,\gamma,\psi)$ ($G\Leftrightarrow(\gamma,\psi)$) be the equivalent pair of classical structures. Then, the given structure is normal iff the classical metric almost contact structures $(F_\pm,Z_\pm,\xi_\pm,\gamma)$ are normal and the following conditions hold:
\begin{equation}\label{indbin0} \begin{array}{l} [Z_+,Z_-]=0,\; L_{Z_-}\xi_++L_{Z_+}\xi_- -d(\gamma(Z_+,Z_-))=i(Z_-)i(Z_+)d\psi, \vspace*{2mm}\\
\zeta_\pm(X_\pm)\circ F_+=\zeta_\pm(X_\pm)\circ F_-=-\zeta_\pm(F_\pm X_\pm),

\vspace*{2mm}\\	 F_\pm[Z_\pm,X_\mp]-[Z_\pm,F_\mp X_\mp] =\frac{1}{2}(F_-\sharp_\gamma\varrho_\pm(X_\mp)-F_+\sharp_\gamma\varrho_\pm(X_\mp)),

\vspace*{2mm}\\	 \varrho_\pm(F_\mp X_\mp)+\varrho_\pm(X_\mp)\circ F_\mp=0,\end{array}\end{equation}
where $X_\pm\in P_\pm=im\,F_\pm$ and $\zeta_\pm(X_\pm),\varrho_\pm(X_\mp)\in\Omega^1(M)$ are defined by
\begin{equation}\label{zetarho} \begin{array}{lcl} \zeta_\pm(X_\pm)&=&i(X_\pm)i(Z_\pm)d\psi\pm(L_{Z_\pm}(\flat_\gamma X_\pm)-\flat_{L_{X_\pm}\gamma}Z_\pm), \vspace*{2mm}\\ \varrho_\pm(X_\mp)&=&\mp\{\flat_{\gamma}[Z_\pm,X_\mp]-i(X_\mp)i(Z_\pm)d\psi\vspace*{2mm}\\ & +&L_{Z_\pm}(\flat_\gamma X_\mp)+i(X_\mp)d\xi_\pm\}. \end{array}\end{equation}
\end{prop}
\begin{proof}
Straightforward computations yield the following Courant brackets (e.g., \cite{V4}):
\begin{equation}\label{CrVpm} \begin{array}{l}
[(X,\flat_{\psi\pm\gamma} X),(Y,\flat_{\psi\pm\gamma} Y)]= ([X,Y],
\flat_{\psi\pm\gamma} [X,Y])
\vspace*{2mm}\\

\hspace*{5mm} +(0,i(Y)i(X)d\psi \pm(L_X(\flat_\gamma Y) -\flat_{L_Y\gamma}X),
\vspace*{2mm}\\

[(X,\flat_{\psi+\gamma} X),(Y,\flat_{\psi-\gamma} Y)]= ([X,Y],
\flat_{\psi} [X,Y])\vspace*{2mm}\\

\hspace*{5mm} +(0,i(Y)i(X)d\psi -L_X(\flat_\gamma Y) -L_Y(\flat_\gamma X)+d(\gamma(X,Y))).\end{array}\end{equation}

We shall use these formulas in the normality conditions of the structure $\mathcal{F}$. If we apply the second formula (\ref{CrVpm}), we see that the first condition (\ref{normaltotal}) is equivalent to the two conditions on the first line of (\ref{indbin0}).

In the second condition (\ref{normaltotal}) we have $\mathcal{X}\in L$ and it is obvious that $L=\tau_+^{-1}(P_+)\oplus \tau_-^{-1}(P_-)$. The pairs that belong to the terms of this direct sum are of the form $\mathcal{X}=(X_\pm,\flat_{\psi\pm\gamma}X_\pm)$, where $X_\pm\in P_\pm$.

Now, firstly, we consider the case of the pairs $\mathcal{Z}_\pm=(Z_\pm,\flat_{\psi\pm \gamma}Z_\pm),\mathcal{X}=(X_\pm,\flat_{\psi\pm \gamma}X_\pm)$.
The first formula (\ref{CrVpm}) gives two conditions that are equivalent to $[\mathcal{Z}_\pm,\mathcal{F}\mathcal{X}]=\mathcal{F}[\mathcal{Z}_\pm,\mathcal{X}]$. The first is $L_{Z_\pm}F_\pm=0$ and it is contained in the classical normality of $(F_\pm,Z_\pm,\xi_\pm)$. The other one is
\begin{equation}\label{normfinal2} \mathcal{F}(0,\zeta_\pm(X_\pm))=(0,\zeta_\pm(F_\pm X_\pm)),\end{equation}
where $\zeta_\pm(X_\pm)$ are defined by (\ref{zetarho}).

Condition (\ref{normfinal2}) may be again decomposed into two components. Indeed, for any $1$-form $\alpha$ one has the decomposition
\begin{equation}\label{descaux}
(0,\alpha)=\frac{1}{2}\{(\sharp_\gamma\alpha,\flat_{\psi+\gamma}\sharp_\gamma\alpha)-
(\sharp_\gamma\alpha,\flat_{\psi-\gamma}\sharp_\gamma\alpha)\},\end{equation}
which allows to express $\mathcal{F}(0,\alpha)$ via $F_\pm$. Accordingly (and using the skew symmetry of $F_\pm$ with respect to $\gamma$, which is equivalent to $F_\pm\sharp_\gamma\alpha=-\sharp_\gamma(\alpha\circ F_\pm)$), we see that the second condition (\ref{normfinal2}) is equivalent to the conditions on the second line of (\ref{indbin0}).
These conditions are invariant under a change $X\mapsto (fX)$ ($f\in C^\infty(M)$); it is easy to check that $\zeta_\pm(X_\pm)$ ($X_\pm\in P_\pm$) is $C^\infty(M)$-linear.

Secondly, we have to consider the condition
\begin{equation}\label{lastc} \begin{array}{l}
[(Z_\pm,\flat_{\psi\pm\gamma}Z_\pm),(F_\mp X_\mp,\flat_{\psi\mp\gamma} (F_\mp X_\mp)]\vspace*{2mm}\\ = \mathcal{F}[(Z_\pm,\flat_{\psi\pm\gamma}Z_\pm),(X_\mp,\flat_{\psi\mp\gamma}X_\mp)].\end{array}
\end{equation}
Using the second formula (\ref{CrVpm}) and (\ref{descaux}), we get
\begin{equation}\label{auxptfinal} \begin{array}{l}[(Z_\pm,\flat_{\psi\pm\gamma}Z_\pm), (X_\mp,\flat_{\psi\mp\gamma}X_\mp)]= \pm\{([Z_\pm,X_\mp],\flat_{\psi\pm\gamma}[Z_\pm,X_\mp])\vspace*{2mm}\\ +\frac{1}{2}[(\sharp_\gamma\varrho_\pm(X_\mp),\flat_{\psi+\gamma}\sharp_\gamma\varrho_\pm(X_\mp))-
(\sharp_\gamma\varrho_\pm(X_\mp),\flat_{\psi-\gamma}\sharp_\gamma\varrho_\pm(X_\mp))]\},\end{array}
\end{equation}
where $\varrho_\pm(X)$ are defined by (\ref{zetarho}).

Now, the left hand side of (\ref{lastc}) is obtained by replacing $X\mapsto F_\mp X_\mp$ in (\ref{auxptfinal}) and, on the other hand, we can express $\mathcal{F}$ in the right hand side of (\ref{lastc}) via $F_\pm$ and (\ref{eqJrond}).
This leads to the expression of the vector and covector components of (\ref{lastc}). For the vector component we get exactly the third line of (\ref{indbin0}). For the covector component we get
\begin{equation}\label{normfinal40} \begin{array}{l}
\flat_{\psi\pm\gamma}(F_\pm[Z_\pm,X_\mp]-[Z_\pm,F_\mp X_\mp])\vspace*{2mm}\\
=\frac{1}{2}(\flat_{\psi-\gamma}F_-\sharp_\gamma\varrho_\pm(X_\mp) -\flat_{\psi+\gamma}F_+\sharp_\gamma\varrho_\pm(X_\mp)+2\varrho_\pm(F\mp X_\mp)).\end{array}
\end{equation}
If we insert the value given by the third line of (\ref{indbin0}) in (\ref{normfinal40}) and look at the upper signs and lower signs separately, we see that (\ref{normfinal40}) may be replaced by the simpler relation that appears in the last line of (\ref{indbin0}).

Finally, let us recall that the last condition (\ref{normaltotal}) is equivalent to the fact that $\mathcal{F}$ is a generalized CRF structure. On the other hand, the first part (formulas (4.38)-(4.44)) of the proof of Proposition 4.3 of \cite{V1} proves that $\mathcal{F}$ is CRF iff the structures $F_\pm$ are classical CRF structures. In our case, we also had
$L_{Z_\pm}F_\pm=0$ (the first condition (\ref{normfinal2})). Hence, generalized normality of $\mathcal{F}$ implies classical normality of $(F_\pm,Z_\pm,\xi_\pm)$ and the latter is part of the conditions of the former. The conclusion of all the above is that, besides the classical normality of $(F_\pm,Z_\pm,\xi_\pm)$, generalized normality exactly requires conditions (\ref{indbin0})).
\end{proof}
\begin{prop}\label{binormexpli} With the same notation as in Proposition \ref{metricnormexplicit}, the structure $(\mathcal{F},\mathcal{Z}_\pm,G)$ is binormal iff
the classical metric almost contact structures $(F_\pm,Z_\pm,\xi_\pm,\gamma)$ are normal and the following conditions hold:
\begin{equation}\label{indbin1} \begin{array}{l} [Z_+,Z_-]=0,\; L_{Z_-}\xi_++L_{Z_+}\xi_-
-d(\gamma(Z_+,Z_-))=i(Z_-)i(Z_+)d\psi, \vspace*{2mm}\\
\zeta_\pm(X_\pm)=0,\; \varrho_\pm(F_\mp X_\mp)+\varrho_\pm(X_\mp)\circ F_\mp=0,\vspace*{2mm}\\ F_\pm[Z_\pm,X_\mp]=\mp\frac{1}{2} F_\pm(\sharp_\gamma\varrho_\pm(X_\mp)),\,[Z_\pm,F_\mp X_\mp] =\mp\frac{1}{2}F_\mp(\sharp_\gamma\varrho_\pm(X_\mp)). \end{array} \end{equation}
\end{prop}
\begin{proof}
The normality of the structure $\mathcal{F}'$ is characterized by (\ref{indbin0}) with the replacements $F_\pm\mapsto\pm F_\pm$, $Z_\pm\mapsto\pm Z_\pm$, $\xi_\pm\mapsto\pm\xi_\pm$, which imply $\zeta_\pm(X)\mapsto\pm \zeta_\pm(X),\varrho_\pm(X)\mapsto\pm\varrho_\pm(X)$. Classical normality and the condition on the first and last line of (\ref{indbin0}) are invariant under these changes.

The conditions on the second line of (\ref{indbin0}) hold for both $\mathcal{F}$ and $\mathcal{F}'$ iff
\begin{equation}\label{3forFF} \zeta_\pm(X_\pm)\circ F_+=0,\;\zeta_\pm(X_\pm)\circ F_-=0,\; \zeta_\pm(F_\pm X_\pm)=0. \end{equation}
The first two conditions (\ref{3forFF}) mean that $\zeta_\pm(X_\pm)|_{P_\pm}=0$. But, from the definition of $\zeta_\pm(X_\pm)$ we get
$$<\zeta_\pm(X_\pm),Z_\pm>=\pm2\gamma([X_\pm,Z_\pm],Z_\pm).$$
Since $(F_\pm,Z_\pm,\xi_\pm)$ are classically normal, we have $L_{Z_\pm}F_\pm=0$ and, in particular, the flow of $Z_\pm$ preserves the image $P_\pm$ of $F_\pm$. Thus, $[X_\pm,Z_\pm]\perp Z_\pm$ and $\zeta_\pm(X_\pm)(Z_\pm)=0$. Hence, (\ref{3forFF}) is equivalent to $\zeta_\pm(X_\pm)=0$, as required in the conclusion of the proposition.

Furthermore, for $\mathcal{F}'$ the condition on the third line of (\ref{indbin0}) becomes
\begin{equation}\label{normfinal4'}
F_\pm[Z_\pm,X]+[Z_\pm,F_\pm X]= \mp(\frac{1}{2}(F_-\sharp_\gamma\varrho_\pm(X) +F_+\sharp_\gamma\varrho_\pm(X)).\end{equation}
We shall add and subtract the third line of (\ref{indbin0}) and (\ref{normfinal4'}), and use the compatibility between $F_\pm$ and $\gamma$. This yields conditions that are equivalent to those on the last line of (\ref{indbin1}).
Notice that these last conditions are invariant under a change $X\mapsto fX$ ($f\in C^\infty(M)$)
because the definition of $\varrho_\pm(X_\pm)$ yields
$$\varrho_\pm(fX_\pm)=f\varrho_\pm(X_\mp) \mp2(Z_\pm f)\flat_\gamma X_\mp.$$
\end{proof}
\begin{example}\label{exbinorm} {\rm Assume that the structure $(\mathcal{F},\mathcal{Z}_\pm,G) \,\Leftrightarrow\,(F_\pm,Z_\pm,\xi_\pm,\gamma,\psi)$
satisfies the following conditions:
\begin{equation}\label{condlastex} \begin{array}{c}
[Z_+,Z_-]=0,\;\gamma(Z_+,Z_-)=const.,\;d\xi_\pm=0,\;\mathcal{N}_{F_\pm}=0,\vspace*{2mm}\\ L_{Z_\pm}F_{\mp}=0,\;L_{Z_\pm}\gamma=0,\;d\psi=0.\end{array}\end{equation}
Then, the structure is a binormal structure.

Indeed, conditions $d\xi_\pm=0,\mathcal{N}_{F_\pm}=0$ imply the normality of the two classical structures involved and the other conditions (\ref{condlastex}) imply $\gamma(Z_+,Z_-)=0$, $L_{Z_\pm}\xi_+=L_{Z_\pm}\xi_-=0$ and ensure the fulfilment of the conditions on the first line of (\ref{indbin1}). Then, modulo (\ref{condlastex}), easy calculations yield $<\zeta_\pm(X_\pm),Z_\pm>=0,<\zeta_\pm(X_\pm),Y_\pm>=0$, i.e., $\zeta_\pm(X_\pm)=0$. Finally, modulo (\ref{condlastex}), the expression (\ref{zetarho}) of $\varrho_\pm(X_\mp)$ reduces to $\varrho_\pm(X_\mp)=\mp2\flat_\gamma[Z_\pm,X_\mp]$ and the remaining conditions (\ref{indbin1}) follow from $L_{Z_\pm}F_{\mp}=0$.

Conditions (\ref{condlastex}) are satisfied for the $(2,1)$-generalized metric almost contact  structure defined by a classical normal metric almost contact structure $(F,Z,\xi)$ such that $d\xi=0$ (then, $\gamma(Z_+,Z_-)=\gamma(Z,Z)=1$).

Another particular case may be obtained as follows. Take $M=N\times(S^1\times S^1)$, where $N$ is a manifold endowed with a normal almost contact structure
$(A\in End(TN),U\in\chi(N),\vartheta\in \Omega^1(N))$ and a compatible Riemannian metric $\sigma$. Let $(t_+,t_-)$ be coordinates of period $1$ on the torus $S^1\times S^1$.
Define $Z_\pm=\partial/\partial t_\pm\in\chi(M)$, $\xi_\pm=dt_\pm\in\Omega^1(M)$, the metric $\gamma$ of $M$ given by
\begin{equation}\label{metricainex}
\gamma=\sigma+(dt_+)^2+(dt_-)^2
\end{equation}
and the endomorphisms $F_\pm\in End(TM)$ given by
\begin{equation}\label{FpmdinA} F_\pm=A+\vartheta\otimes Z_\mp -\xi_\mp\otimes U,
\end{equation}
equivalently,
\begin{equation}\label{Fpmdetal} \begin{array}{l} F_+|_{im\,A}=A,\,F_+U=Z_-,\,F_+Z_+=0,\,F_+Z_-=-U\vspace*{2mm}\\ F_-|_{im\,A}=A,\,F_-U=Z_+,\,F_-Z_+=-U,\,F_-Z_-=0.\end{array}\end{equation}

One can straightforwardly check that the structures $(F_\pm,Z_\pm,\xi_\pm,\gamma)$ are metric almost contact structures on $M$. These structures obviously satisfy the first three conditions (\ref{condlastex}) (in particular, $\gamma(Z_+,Z_-)=0$). By applying the Lie derivatives $L_{Z_\pm}$ to formulas (\ref{metricainex}), (\ref{FpmdinA}) we also get sixth and fifth condition (\ref{condlastex}) as well as the conditions $L_{Z_\pm}F_\pm=0$. The last condition (\ref{condlastex}) holds if we choose a closed $2$-form $\psi$.

The only remaining condition is $\mathcal{N}_{F_\pm}=0$ and its role is to ensure the normality of $(F_\pm,Z_\pm,\xi_\pm,\gamma)$. Since $L_{Z_\pm}F_\pm=0$, the required normality holds if the structures $(F_\pm,Z_\pm,\xi_\pm,\gamma)$ are of the CR type, equivalently, the Nijenhuis tensor $\mathcal{N}_{F_\pm}$ vanishes on arguments in $im\,F_\pm$, respectively.

From (\ref{Fpmdetal}), we get $im\,F_\pm=im\,A\oplus span\{U,Z_\mp\}$. $\mathcal{N}_{F_\pm}(U,Z_\mp)=0$ is trivial. If $X\in im\,A$, $\mathcal{N}_{F_\pm}(X,U)=0,\mathcal{N}_{F_\pm}(X,Z_\mp)=0$ follows from $L_UA=0$ (use the fact that the flow of $U$ preserves the image of $A$), which holds because of the normality of the structure $(A,U,\vartheta)$. The last required condition is $\mathcal{N}_{F_\pm}(X,Y)=0$, $\forall X,Y\in im\,A$, equivalently, $A$ is of the CR type. This fact is also implied by the normality of the structure $(A,U,\vartheta)$.

Therefore, if $N$ is a normal metric almost contact manifold, $N\times(S^1\times S^1)$ is a binormal metric $(2,1)$-generalized almost contact manifold.}\end{example}

More examples of binormal (coK\"ahler) structures were given in Section 5 of \cite{GT}. Moreover, the main result of \cite{GT} is that the product of two generalized metric contact manifolds with the naturally induced generalized almost Hermitian structure is a generalized K\"ahler manifold iff the factors are generalized coK\"ahler.

We shall finish by references to the subject of defining a natural notion of generalized Sasakian structure already discussed in \cite{{V1},{V3},{Sek}}.

Remember that a classical Sasakian structure is a metric almost contact structure $(F,Z,\xi,\gamma)$ such that the pair $(J,e^t(\gamma+dt^2))$, where $J$ is defined by (\ref{JF}), is a K\"ahler structure on $M\times\mathds{R}$. Now, consider a metric $(2,1)$-generalized almost contact structure $(\mathcal{F},\mathcal{Z}_\pm,G)$ on $M$. Then, as seen during the preparations for and proof of Proposition \ref{propbinorm}, we have the second structure $(\mathcal{F}',\mathcal{Z}'_\pm,G)$, the corresponding, commuting, generalized almost complex structures $\mathcal{J},\mathcal{J}'$ and the generalized Riemannian metric $\tilde{G}$ with $\tilde{\mathcal{G}}=-\mathcal{J\circ\mathcal{J}'}$ on $M\times\mathds{R}$.
Of course, $(\tilde{G},\mathcal{J})$ is a generalized almost Hermitian structure.
The last part in the proof of Proposition \ref{propbinorm} implies
$$
\tilde{G}=G+(\flat_g\mathcal{T}_+)\otimes(\flat_g\mathcal{T}_+) +(\flat_g\mathcal{T}_-)\otimes(\flat_g\mathcal{T}_-)$$
and shows that the corresponding $\pm1$-eigenbundles on $TM\times\mathds{R}$ are
\begin{equation}\label{tildeVpm}
\tilde{V}_\pm=V_\pm\oplus span\{\mathcal{T}_\pm\}.\end{equation}

As usually, we consider the $G$-equivalent pair $(\gamma,\psi)$, where $\gamma$ is a Riemannian metric on $M$ and $\psi\in\Omega^2(M)$. We recover the similar pair $(\tilde{\gamma},\tilde{\psi})$ of $\tilde{G}$ on$M\times\mathds{R}$ as follows. For each $\tilde{X}\in T(M\times\mathds{R})$, we find the unique $1$-form $\tilde{\alpha}_\pm \in T^*(M\times\mathds{R})$ such that $(\tilde{X},\tilde{\alpha}_\pm)\in \tilde{V}_\pm$. Then, we obviously have
$$\flat_{\tilde{\psi}}\tilde{X}=\frac{1}{2}(\tilde{\alpha}_++\tilde{\alpha}_-),\; \flat_{\tilde{\gamma}}=\frac{1}{2}(\tilde{\alpha}_+-\tilde{\alpha}_-).$$
Now, formula (\ref{tildeVpm}) shows that, for $\tilde{X}=X\in TM$, $\tilde{\alpha}_\pm=\alpha_\pm= \flat_{\psi\pm\gamma}X$ and, for $\tilde{X}=\partial/\partial t$, $\tilde{\alpha}_\pm=\pm dt$. Accordingly, we get
$$\tilde{\gamma}=\gamma+dt^2,\;\tilde{\psi}=\psi.$$

On the other hand, several authors have defined the notion of a conformal change $\mathcal{C}_\tau\in End(\mathbf{T}M)$, $\tau\in C^\infty(M)$, by $\mathcal{C}_\tau(X,\alpha)=(X,e^\tau\alpha)$ (e.g., \cite{VAnn}). Accordingly, on $M\times\mathds{R}$, we have the endomorphisms
$$\mathcal{J}_t=\mathcal{C}_{-t}\circ\mathcal{J}\circ\mathcal{C}_t,\,
\mathcal{J}'_t=\mathcal{C}_{-t}\circ\mathcal{J}'\circ\mathcal{C}_t,\,
\tilde{\mathcal{G}}_t=\mathcal{C}_{-t}\circ\tilde{\mathcal{G}}\circ\mathcal{C}_t,$$
which define a generalized almost Hermitian structure $(G_t,\mathcal{J}_t,\mathcal{J}'_t)$ such that $\tilde{\mathcal{G}}_t=-\mathcal{J}_t\circ\mathcal{J}'_t$.

These considerations show that it is natural to give the following definition.
\begin{defin}\label{Sasgen} {\rm A {\it $(2,1)$-generalized Sasakian structure} is a $(2,1)$-generalized metric almost contact structure $(\mathcal{F},\mathcal{Z}_\pm,G)$ on $M$ such that the corresponding structure $(G_t,\mathcal{J}_t,\mathcal{J}'_t)$ of $M\times\mathds{R}$ is a generalized K\"ahler structure, i.e., the structures $\mathcal{J}_t,\mathcal{J}'_t$ are integrable.}\end{defin}

The classical pair that corresponds to the generalized metric
$\tilde{\mathcal{G}}_t$ is \cite{VAnn}
$$\tilde{\gamma}_t=e^t\tilde{\gamma},\,\tilde{\psi}_t=e^t\tilde{\psi}.$$
Moreover, the pair of almost complex structures $\tilde{J}_{t\pm}$ of $M\times\mathds{R}$, which together with $\tilde{\gamma}_t,\tilde{\psi}_t$ define the generalized almost Hermitian structure $(\tilde{G}_t,\mathcal{J}_t)$, is just $\tilde{J}_{t\pm}=\tilde{J}_{\pm}$ \cite{VAnn}.
This shows that Definition \ref{Sasgen} yields the same generalized Sasakian structures as the definition of \cite{Sek}, while the definition in \cite{V3}, where $\tilde{\psi}$ is $\psi+\kappa\wedge dt$ for any $\kappa\in\Omega^1(M)$, is slightly more general. The conditions that characterize generalized Sasakian structures were given in Theorem 24 of \cite{V3}.

{\small Department of Mathematics, University of Haifa, Israel.\\ E-mail: vaisman@math.haifa.ac.il}
\end{document}